\documentclass[12pt,oneside, reqno]{article} 
\usepackage{amsmath}
\usepackage{amssymb}
\usepackage{graphicx}

\usepackage{mathrsfs}
\usepackage{pb-diagram}

\graphicspath{{./Figure/}}

\newtheorem{theorem}{Theorem}[section]

\newtheorem{definition}{Definition}[section]

\newtheorem{example}{Example}[section]
\numberwithin{equation}{section}

\DeclareMathOperator{\rot}{rot}
\DeclareMathOperator{\Rot}{Rot}
\def\stretchx{\Bumpeq{\!\!\!\!\!\!\!\!{\longrightarrow}}}



\begin{document}

\title{{\Large Chaotic dynamics in the presence of medical malpractice litigation:\\ a topological proof via linked twist maps for two evolutionary game theoretic contexts} 
}

\vspace{-2mm}

\begin{author}
{Marina Pireddu\footnote{University of Milano\,-\,Bicocca, U5 Building, Via Cozzi 55, 20125 Milano, Italy. \qquad Tel.: +39\,0264485767. E-mail address: marina.pireddu@unimib.it}\\
{\footnotesize{Dept. of Mathematics and its
Applications, University of Milano\,-\,Bicocca, Milano, Italy.}}}
\end{author}
\date{}
\maketitle

\vspace{-9mm}

\begin{abstract}
\noindent
In the present work we reconsider the evolutionary game theoretic models by Antoci et al. (2016, 2018) describing the dynamic outcomes arising from the interactions between patients and physicians, whose behavior is subject to clinical and legal risks. In particular, Antoci et al. (2016) analyzed the case of positive defensive medicine, while Antoci et al. (2018) dealt with the case of negative defensive medicine.
We show that, when the models admit a nonisochronous center, it is possible to prove the existence of chaotic dynamics for the Poincar\'e map associated with those systems via the method of Linked Twist Maps (LTMs). To such aim we exploit in both frameworks, using a similar rationale, the periodic dependence on time of a model parameter influencing the position of the center and describing some risk associated with certain medical interventions, whose seasonal variation is empirically grounded.
We also provide the ecological interpretation of the same model, proposed by Harvie et al. (2007) and connected with intraspecific competition and environmental carrying capacity in 
predator-prey settings. In this case, it is sensible to assume a seasonal variation both for the carrying capacities and for the 
intrinsic growth rates of the two populations. Although such parameters influence the shape of the orbits, but do not affect the center position, we show that it is still possible to prove the existence of chaotic dynamics for the associated Poincar\'e map via the LTMs technique dealing with a different geometrical configuration for orbits in the phase plane.
\end{abstract}

\vspace{2mm}


\vspace{2mm}


\section{Introduction}\label{sec-1}
The aim of the present contribution consists in providing a topological proof of the existence of chaotic dynamics for a planar Hamiltonian system considered in different formulations and with different interpretations in the literature, by using the ``Linked Twist Maps'' (from now on, LTMs) technique. More precisely, the model that we are going to analyze has been proposed in two evolutionary game theoretic frameworks
by Antoci et al. in \cite{Anea-16,Anea-18} to describe the dynamic outcomes arising from the interactions between patients and physicians, whose behavior is subject to clinical and legal risks. Namely, both works deal with defensive medicine, i.e., the deviation from good medical practice motivated by the threat of liability claims, investigated e.g. in \cite{KeMC-96,Stea-05,TaBa-78}. In particular, in \cite{Anea-16}  the focus is on positive defensive medicine (see for instance \cite{Su-95}), according to which additional services are offered to discourage patients from
filing malpractice claims, or to convince the legal system that the standard of care was
met. On the other hand, the authors in \cite{Anea-18} consider
negative defensive medicine, that is the case in which physicians tend to avoid a source
of legal risk e.g. by adopting safer but less effective treatments (cf. \cite{Fe-12}).
A biological interpretation of the setting considered in \cite{Anea-16,Anea-18}, connected
with intraspecific competition and environmental carrying capacity
in predator-prey models, has been briefly suggested in \cite{Haea-07}, where however the focus is on the celebrated growth-cycle model by Goodwin \cite{Go-67,Go-72}, describing the dynamics of the wage share of output and the employment proportion. Actually, Harvie et al. in \cite{Haea-07} propose a system of differential equations in which each variable has both a positive and a
negative effect on its own growth rate. We will focus just on the latter, in order to obtain again the same formulation considered in \cite{Anea-16,Anea-18}.\\
The LTMs technique, that we are going to use in order to prove the existence of complex dynamics in the above described frameworks, is based on the Stretching Along the Paths (henceforth, SAP) method, i.e., the topological method for the search of fixed points
and periodic points for continuous maps defined on sets homeomorphic to the unit cube in finite
dimensional Euclidean spaces that expand the arcs along one
direction, developed in the planar case in \cite{PaZa-04a,PaZa-04b} and extended to the $N$-dimensional framework
in \cite{PiZa-07}. The context of LTMs represents a geometrical framework in which it is possible to employ the SAP method in order to detect complex dynamics, as first shown in \cite{PaZa-08,PiZa-08}. Further applications of the SAP method to planar LTMs contexts have been provided e.g. in \cite{BuZa-09,BuZa-10,PaZa-09}, while a biological application to a three-dimensional LTMs context has been proposed in  \cite{RZZa-14}.\footnote{We also recall the related works \cite{RZZa-15,ZaZa-09}, in which
 the existence of complex dynamics has been proved in a 2D and in a 3D continuous-time framework, respectively, by means of the SAP technique, without relying on the LTMs geometry.} Usual assumptions on the twist mappings 
(see e.g. \cite{BuEa-80,Pr-83,Pr-86}) concern, among others, their smoothness, preservation of the Lebesgue measure and
monotonicity of the angular speed with respect to the radial coordinate, also in view of checking an hyperbolicity condition, needed in order to ensure the existence of Smale horseshoes (cf. \cite{De-78}). On the other hand, since our approach is purely topological, we just need a twist condition on the boundary of the two linked annuli, similar to the Poincar\'e-Birkhoff fixed point theorem. Namely, by a linked twist map we mean the composition of two twist maps, acting each on one of the two linked annuli, which, in the two-dimensional case, cross along two (or more) planar sets homeomorphic to the unit square, that we call topological rectangles.\\ 
In our applications of LTMs, we will consider Hamiltonian systems with a nonisochronous center, in which the position of the center or the shape of the orbits vary when modifying one of the model parameters. In particular, we will modify parameters for which it is sensible to assume that they alternate in a periodic fashion between two different values, e.g. due to a seasonal effect. In this manner, we obtain two conservative systems, the original and the perturbed ones, and for each of them we can consider an annulus composed of energy levels. Under suitable conditions on the orbits, depending on the geometric configuration we are dealing with, the two annuli can cross in two or more generalized rectangles, and in that case we call them linked annuli.
In such context, the LTMs technique consists in finding two linked annuli, whose intersection sets contain chaotic sets for the Poincar\'e map obtained as composition of the Poincar\'e maps associated with the original system and the perturbed one. 
We stress that the nonisochronicity of the center is crucial in the just described procedure, because it implies that the orbits composing the linked annuli are run with a different speed, so that the Poincar\'e maps produce a twist effect on the annuli.
The consequent deformation produced on the generalized rectangles grows with the passing of time. Hence, if the switching times between the regimes governed by the original system and by the perturbed one are large enough, those generalized rectangles are transformed by the Poincar\'e maps into spiral-like sets, intersecting several times the same generalized rectangles, and the stretching along the path property required by the SAP method is fulfilled, guaranteeing the presence of chaotic sets inside the generalized rectangles.\\
We will use the LTMs technique both when dealing with the evolutionary game theoretic frameworks describing the dynamic effects of positive defensive medicine, proposed in \cite{Anea-16}, and of negative defensive medicine, investigated in \cite{Anea-18}, as well as when considering the context connected with intraspecific competition and environmental carrying capacity in predator-prey models, mentioned in \cite{Haea-07}.
Although all such settings are variants of the same model, due to the dissimilar meaning attached to the parameters in the three analyzed contexts, each time it will be sensible to periodically perturb a different parameter, and this in turn will generate a peculiar geometrical configuration, in which the LTMs method can be applied in a specific manner. Namely, in the frameworks from \cite{Anea-16, Anea-18} we will exploit the periodic dependence on time of the model parameters describing the risk associated with certain medical interventions, whose seasonal variation is empirically grounded, considering seasonality in hypertension (see e.g. \cite{Atea-08,Deea-12})
and its connection with perioperative adverse events (cf. for instance \cite{Hoea-04,Liuea-16}).
Despite the common rationale for assuming a periodic variation on those two parameters in the settings from \cite{Anea-16,Anea-18}, we stress that modifying the former or the latter produces a different effect on the position of the center. Moreover, orbits are run clockwise in the framework in \cite{Anea-16} and counterclockwise in the framework in \cite{Anea-18}. Also this difference will affect the proofs of the corresponding results about LTMs, in which we need to count the laps completed by suitable paths around the centers.
As concerns the biological setting, encompassing logistic terms which take into account intra-species interactions and the role of environmental resources, according e.g. to \cite{BoSt-15,NiGu-76}
 it is sensible to assume a seasonal variation both for the carrying capacities and for the 
intrinsic growth rates of the two populations. Even if such parameters influence the shape of the orbits, and raising the value of carrying capacities also enlarges the region in which orbits may lie, neither carrying capacities nor intrinsic growth rates affect the center position. Nonetheless, we show that it is still possible to prove the existence of chaotic dynamics for the associated Poincar\'e map via the LTMs technique dealing with a different geometrical configuration for orbits in the phase plane, in agreement with the results obtained in other contexts e.g. in \cite{BuZa-09,PaZa-13}. Regarding the nonisochronicity of the centers, it was proven in \cite{Sc-85,Sc-90} that the period of the orbits increases with the energy level in the settings that we are going to analyze.\footnote{In this respect, we also mention \cite{Maea-16}, where results about the period of small and large cycles have been obtained for a wide class of Hamiltonian systems, encompassing those here considered.} We finally stress that, like it happened e.g. in \cite{PiZa-08}, where the effect produced by a periodic harvesting on the original predator-prey model was investigated, also our results about the existence of chaotic dynamics are robust with respect to small perturbations, in $L^1$ norm, in the coefficients of the considered settings.\\
The remainder of the paper is organized as follows. In Section \ref{sec-2} we recall the main definitions and results connected with the LTMs method. In Section \ref{sec-3} we introduce the first two versions of the model that we are going to analyze, as presented in \cite{Anea-16,Anea-18}, respectively, and we prove that the Poincar\'e maps associated with them may generate chaotic dynamics via the method of LTMs, looking at the geometrical configurations of their orbits in the phase plane. In Section \ref{sec-4} we illustrate the ecological interpretation given in \cite{Haea-07} of the model considered in Section \ref{sec-3} and we describe an alternative way of applying the LTMs technique.
In Section \ref{sec-5} we briefly discuss our results and conclude.

\section{Recalling the Linked Twist Maps framework}\label{sec-2}

As explained in the Introduction, our theoretical starting point is given by the Stretching Along
the Paths (henceforth, SAP) method, developed in the planar case in \cite{PaZa-04a,PaZa-04b} and extended to the $N$-dimensional framework
in \cite{PiZa-07}. We shall see below that the SAP method is based on the SAP relation in \eqref{sapr}.
The context of ``Linked Twist Maps'' (LTMs, from now on) represents a geometrical framework in which it is possible to employ the SAP method in order to detect complex dynamics, as first shown in \cite{PaZa-08,PiZa-08}. In more detail, 
by a linked twist map we mean the composition of two twist maps, each acting on one of two annuli, which, in the two-dimensional case, cross along two (or more) planar sets homeomorphic to the unit square, that we call topological rectangles. For the two maps, we assume that they are homeomorphisms and that, like in the Poincar\'e-Birkhoff fixed point theorem, produce a twist effect on the boundary of the two annuli, leaving the boundary invariant. However, our approach is purely topological and for the maps we do not require neither area-preserving properties, nor the monotonicity of the twist with respect to the radial coordinate.\\ 
For brevity's sake, in what follows we will recall just the definitions and the results about the SAP relation that are necessary in view of our planar applications in Sections \ref{sec-3} and \ref{sec-4}. Further details and more general formulations can be found e.g. in \cite{Paea-08,PiZa-08}. Moreover, a three-dimensional LTMs context has been proposed in \cite{RZZa-14}.\\
A \textit{path} in $\mathbb R^2$ is a continuous map $\gamma:
[0,1]\to\mathbb R^2$ and we set $\overline{\gamma}:=\gamma([0,1]).$  
By a \textit{generalized rectangle} 
we mean a set ${\mathcal R}\subseteq\mathbb R^2$ which is homeomorphic to the
unit square $I^2:=[0,1]^2,$ through a homeomorphism $h: {\mathbb R}^2\supseteq I^2 \to \mathcal R\subseteq \mathbb R^2.$
We also set
$${\mathcal R}^{-}_{l}:= h([x_1 = 0])\,,\quad
{\mathcal R}^{-}_{r}:= h([x_1 = 1])$$ and call them the \textit{left} and
the \textit{right} sides of $\mathcal R,$ respectively. 
Setting $\mathcal R^{-}:= \mathcal R^{-}_{l}\cup \mathcal R^{-}_{r},$
we call the pair
$${\widetilde{\mathcal R}}:= (\mathcal R, \mathcal R^-)$$
an {\textit{oriented rectangle}} \textit{of $\mathbb R^2$}.\\
We are now in position to recall the definition of the {\it stretching along the paths} relation for maps between oriented rectangles.\\
Given ${\widetilde{\mathcal A}}:=
({\mathcal A},{\mathcal A}^-)$ and ${\widetilde{\mathcal B}}:=
({\mathcal B},{\mathcal B}^-)$ oriented rectangles of $\mathbb R^2,$ let $F: \mathcal A\to \mathbb R^2$ be a function
and ${\mathcal K}\subseteq {\mathcal A}$
be a compact set. We say that \textit{$({\mathcal K},F)$ stretches
${\widetilde{\mathcal A}}$ to ${\widetilde{\mathcal B}}$ along the
paths}, and write
\begin{equation}\label{sapr}
({\mathcal K},F): {\widetilde{\mathcal A}} \stretchx {\widetilde{\mathcal B}},
\end{equation}
if
\begin{itemize}
\item{} \; $F$ is continuous on ${\mathcal K}\,;$
\vspace{-2mm} 
\item{} \; for every path $\gamma: [0,1]\to {\mathcal A}$ with
$\gamma(0)$ and $\gamma(1)$ belonging to different components of ${\mathcal A}^-,$ there exists $[t',t'']\subseteq [0,1]$
such that $\gamma([t',t''])\subseteq {\mathcal K},$ $F(\gamma([t',t'']))\subseteq {\mathcal B},$ with $F(\gamma(t'))$ and
$F(\gamma(t''))$ belonging to different components of ${\mathcal B}^-.$
\end{itemize}
In the special case in which ${\mathcal K}={\mathcal A},$ we simply write $F: {\widetilde{\mathcal A}} \stretchx {\widetilde{\mathcal B}}.$

\smallskip

In our applications of LTMs, we will consider Hamiltonian systems with a nonisochronous center, in which the position of the center or the shape of the orbits vary when modifying one of the model parameters. In particular, we will modify parameters for which it is sensible to assume that they alternate in a periodic fashion between two different values, e.g. due to a seasonal effect. In this manner, we obtain two conservative systems and for each of them we can consider an annulus composed of energy levels. Under suitable conditions on the orbits, depending on the geometric configuration we are dealing with, the two annuli can cross in two or more generalized rectangles, that we orientate by suitably choosing how to name (as left and right) two among the arcs of orbits composing their boundary.
As we shall see in the next sections, the choice depends on the relative position of the generalized rectangles and on whether orbits are run clockwise or counterclockwise. The involved functions are the Poincar\'e maps associated with the two systems, and thus they are homeomorphisms.

The main result that we shall use in Sections \ref{sec-3} and \ref{sec-4} reads as follows:

\begin{theorem}\label{th}
Let $F: \mathbb R^2\supseteq D_{F}\to \mathbb R^2$ and
$G: \mathbb R^2\supseteq D_{G}\to \mathbb R^2$ be continuous maps defined on the sets $D_{F}$ and $D_{G},$ respectively. Let also
${\widetilde{\mathcal A}} := ({\mathcal A},{\mathcal A}^-)$ and
${\widetilde{\mathcal B}} := ({\mathcal B},{\mathcal B}^-)$ be oriented rectangles of $\mathbb R^2.$
Suppose that the following conditions are satisfied:
\begin{itemize}
\item[$\quad (C_F)\;\;$] There are (at least) two disjoint compact sets ${\mathcal H_0},\,{\mathcal H_1}\,\subseteq {\mathcal A}\cap D_{F}$
such that
$\displaystyle{({\mathcal H}_i,F): {\widetilde{\mathcal A}}
\stretchx\, {\widetilde{\mathcal B}}},$ for $i=0,1\,;$
\\
\item[$(C_G)\;\;$] ${\mathcal B}\subseteq D_{G}$ and
$G\,: {\widetilde{\mathcal B}} \stretchx {\widetilde{\mathcal A}}\,.$
\end{itemize}
Then if the map $\Phi:=G\circ F$ is continuous and injective on the set 
${\mathcal H}:=({\mathcal H}_0\cup{\mathcal H}_1)\cap{F}^{-1}(\mathcal B),$ setting
\begin{equation}\label{xin}
{X}_{\infty}:=\bigcap_{n=-\infty}^{\infty}\Phi^{-n}(\mathcal H),
\end{equation}
there exists a nonempty compact set
$${X}\subseteq {X}_{\infty} \subseteq {\mathcal H},$$
on which the following properties are fulfilled:

\begin{itemize}
\item[$(i)$] ${X}$
is invariant for $\Phi$ (i.e., $\Phi(X) = X$);

\item[$(ii)$] $\Phi\!\!\restriction_{X}$ is semi-conjugate to the two-sided Bernoulli shift on two symbols, i.e.,
there exists a continuous map $\pi$ from ${X}$ onto $\Sigma_2:=\{0,1\}^{\mathbb Z},$ endowed with the distance
\begin{equation*}
\hat d(\textbf{s}', \textbf{s}'') := \sum_{i\in {\mathbb Z}} \frac{|s'_i - s''_i|}{2^{|i| + 1}}\,,
\end{equation*}
for $\textbf{s}'=(s'_i)_{i\in {\mathbb Z}}$ and
$\textbf{s}''=(s''_i)_{i\in {\mathbb Z}}\in \Sigma_2\,,$
such that the diagram
\begin{equation*}
\begin{diagram}
\node{{X}} \arrow{e,t}{\Phi} \arrow{s,l}{\pi}
      \node{{X}} \arrow{s,r}{\pi} \\
\node{\Sigma_2} \arrow{e,b}{\sigma}
   \node{\Sigma_2}
\end{diagram}
\end{equation*}
commutes, i.e. $\pi\circ\Phi=\sigma\circ\pi,$ where $\sigma:\Sigma_2\to\Sigma_2$ is the Bernoulli
shift defined as $\sigma((s_i)_i):=(s_{i+1})_i,\,\forall
i\in\mathbb Z\,;$

\item[$(iii)$] the set of the periodic points of $\Phi\restriction_{{X}_{\infty}}$ is dense in ${X}$
and the preimage $\pi^{-1}(\textbf{s})\subseteq {X}$ of
every
$k$-periodic sequence $\textbf{s} = (s_i)_{i\in {\mathbb N}}\in \Sigma_2$
contains at least one $k$-periodic point.
\end{itemize}

Furthermore, from conclusion $(ii)$ it follows that:

\begin{itemize}
\item[$(iv)$] $$h_{\rm top}(\Phi)\ge h_{\rm top}(\Phi\restriction_{X})\geq h_{\rm top}(\sigma) = \log(2),$$
where $h_{\rm top}$ is the topological entropy;

\item[$(v)$] there exists a compact invariant set $\Lambda\subseteq {X}$ such that $\Phi\vert_{\Lambda}$ is
semi-conjugate to the two-sided Bernoulli shift on two symbols, topologically transitive and displays sensitive dependence on initial conditions.
\end{itemize}
\end{theorem}
\begin{proof}
The key step consists in showing that
\begin{equation}\label{2} 
({\mathcal H}_i\cap F^{-1}(\mathcal B),\Phi): {\widetilde{\mathcal A}}\stretchx {\widetilde{\mathcal A}}\,,\;\;\, i=0,1.
\end{equation}
See Theorem 3.1 in \cite{PaZa-08} for a verification of a more general version of this property. 
The desired conclusions then follow by Lemma 3.2, Lemma 3.3 and the discussion after Definition 1.1 in \cite{PiZa-08}.
\end{proof}$\hfill\square$

\newpage

\noindent
Recalling:
\begin{itemize}
\item the definition of a map inducing chaotic dynamics on two symbols in a subset $\mathcal D$ of its domain (cf. Definition 2.2 in \cite{Meea-09}), according to which every two-sided sequence on two symbols - say $0$ and $1$ - is realized through the iterates of the map, jumping between two disjoint compact subsets - say $\mathcal K_0$ and $\mathcal K_1$ - of $\mathcal D,$ and periodic sequences of symbols are reproduced by periodic orbits of the map; 
\item Theorem 2.3 in \cite{Meea-09}, which tells that any time the stretching relation \eqref{sapr} is fulfilled for a continuous function with respect to two disjoint compact sets, then that function induces chaotic dynamics on two symbols\footnote{Namely, as explained in \cite{Meea-09}, the SAP method allows to prove the presence of chaotic dynamics for a continuous map when verifying the validity of the SAP relation in \eqref{sapr} with respect to two suitable disjoint compact subsets of its domain.}, 
\end{itemize}
we can conclude that, when conditions $(C_F)$ and $(C_G)$ in Theorem \ref{th} are satisfied, the composite function $G\circ F$ induces chaotic dynamics on two symbols in $\mathcal A,$ since condition \eqref{2} is fulfilled. On the other hand, according to Theorem 2.2 in \cite{Meea-09}, if an injective map induces chaotic dynamics on two symbols in a set, then it possesses all the features listed in Theorem \ref{th}.\footnote{If the map $\Phi$ were not injective, rather than ${X}_{\infty}$ in \eqref{xin} we should introduce ${X}_{\infty}^{+}:=\bigcap_{n=0}^{\infty}\Phi^{-n}(\mathcal H)$ and then the properties listed in Theorem \ref{th} would hold true replacing $\Sigma_2:=\{0,1\}^{\mathbb Z}$ with $\Sigma_2^{+}:=\{0,1\}^{\mathbb N}.$ This means, in particular, that we could derive a weaker conclusion than $(ii)$ in Theorem \ref{th}, establishing a semi-conjugacy between the map $\Phi$ and the one-sided Bernoulli shift $\sigma^{+}((s_i)_i):=(s_{i+1})_i,\,\forall i\in\mathbb N\,,$ rather than with the two-sided Bernoulli shift. See Theorem 2.2 in \cite{Meea-09} for the precise statement and for a proof. More generally, we refer the interested reader to \cite{Meea-09,PiZa-08} for further details, as well as for a discussion on the various notions of chaos used in the literature and on the relationships among them.}\\
Hence, we can summarize the statement of Theorem \ref{th} by saying that, when conditions $(C_F)$ and $(C_G)$ therein are satisfied, the composite function $G\circ F$ induces chaotic dynamics on two symbols in $\mathcal A,$ knowing that from this fact it follows that
all the properties listed in Theorem \ref{th} are fulfilled for $G\circ F,$ also in regard to the existence of periodic points. We will use this reformulation of Theorem \ref{th} in Sections \ref{sec-3} and \ref{sec-4} when dealing with the composition of Poincar\'e maps associated with different systems (cf. e.g. the statement of Theorem \ref{app18}). We recall that a classical approach - see \cite{Kr-68} - to show the existence of periodic
solutions (harmonics or subharmonics) of systems of first order ODEs with periodic coefficients, under the 
assumption of uniqueness of the solutions for the Cauchy problems, is based on
the search of the fixed points or periodic points for the associated Poincar\'e map.\\ 
We close the present preliminary section by completing the explanation, started just before Theorem \ref{th}, of what the LTMs method consists in, referring the reader to Subsections \ref{31} and \ref{32} and to Section \ref{sec-4} for some concrete applications of it, as well as for a few graphical illustrations of possible geometrical configurations related to such framework (see in particular Figures \ref{nef}--\ref{biofk}, connected with the numerical examples). Given a Hamiltonian system with a nonisochronous center, led by the economic or biological interpretation of the considered context, we perturb it by modifying a parameter, for which it is sensible to assume that it  alternates in a periodic fashion between two different values. The LTMs method consists in proving the presence of chaotic dynamics for the Poincar\'e map obtained as composition of the Poincar\'e maps associated with the original system and the perturbed one,
by finding two linked annuli, each composed by orbits of one of the two systems, to which it is possible to apply Theorem \ref{th}. Notice that it would not be possible to find two linked annuli composed by orbits of the same Hamiltonian system. More precisely, the existence of chaos is shown by applying Theorem \ref{th}, i.e., by checking that conditions $(C_F)$ and $(C_G)$ therein are satisfied when taking as $F$ and $G$ the Poincar\'e maps associated with the two systems, and by choosing as $\mathcal A$ and $\mathcal B$ two among the generalized rectangles in which the considered linked annuli cross.\footnote{Namely, in the frameworks considered in Subsections \ref{31} and \ref{32} the linked annuli cross in two generalized rectangles (see Figures \ref{nef} and \ref{16-la}), while in Section \ref{sec-4} we find four generalized rectangles as intersections sets between two linked annuli (see Figures \ref{bioer} and \ref{bioek}). Indeed, the precise definition of linked annuli may vary according to the geometrical configuration of the considered framework (cf. Definitions \ref{li} and \ref{lim}).} We stress that the nonisochronicity of the center is crucial in the above described procedure, because it implies that the orbits composing the linked annuli are run with a different speed, so that the Poincar\'e maps produce a twist effect on the linked annuli, although closed orbits are invariant under the action of the Poincar\'e maps.
The consequent deformation produced on the generalized rectangles grows with the passing of time. Hence, if the switching times between the regimes governed by one of the two systems are large enough, those generalized rectangles are transformed by the Poincar\'e maps into spiral-like sets, intersecting several times the same generalized rectangles, allowing to check the stretching relations in $(C_F)$ and $(C_G),$
so that the presence of complex dynamics is guaranteed by Theorem \ref{th}.\\
This is the methodology that we are going to use in Section \ref{sec-3}, dealing with the evolutionary game theoretic frameworks describing the dynamic effects of positive defensive medicine, proposed in \cite{Anea-16}, and of negative defensive medicine, investigated in \cite{Anea-18}, as well as in Section \ref{sec-4}, where we consider a context introduced in \cite{Haea-07} and connected with intraspecific competition and environmental carrying capacity in predator-prey models.
We stress that all the frameworks that we will consider along the manuscript are variants of the same model. However, due to the different meaning attached to the parameters in the analyzed contexts, each time it will be sensible to periodically perturb a different parameter, and this in turn will generate a peculiar geometrical configuration, in which the LTMs method can be applied in a specific manner.

\section{Two medical malpractice litigation frameworks}\label{sec-3}
In the present section we introduce the settings considered in \cite{Anea-16,Anea-18} and we explain how to prove the existence for each of them of chaotic dynamics via the method of LTMs. Due to the large number of parameters involved in the models, we will just describe the strictly needed aspects, referring the interested reader to \cite{Anea-16,Anea-18} for further details.\\
Both works deal with defensive medicine, i.e., the deviation from good medical practice motivated by the
threat of liability claims, investigated e.g. in \cite{KeMC-96,Stea-05,TaBa-78}. In particular, Antoci et al. in \cite{Anea-16} propose an evolutionary game theoretic model to investigate the dynamic effects of positive defensive medicine, in which, according e.g. to \cite{Su-95}, additional services are offered to discourage patients from
filing malpractice claims, or to convince the legal system that the standard of care was
met. On the other hand, still in an evolutionary game theoretic model, the authors in \cite{Anea-18} consider
negative defensive medicine, that is the case in which physicians tend to avoid a source
of legal risk e.g. by adopting safer but less effective treatments (cf. \cite{Fe-12}). An important example of this phenomenon, analyzed in 
\cite{Duea-99, Duea-01,Loea-93}, is given by the excessive number of Cesarean sections, which can often be unnecessary. We stress that an extension of the frameworks considered in \cite{Anea-16,Anea-18} has been proposed in \cite{Anea-19}, where it is assumed that physicians have the possibility of insuring against liability claims, so that the system becomes four-dimensional, with the additional variables being 
 the share of the population of physicians who buy a malpractice insurance and the cost of the insurance policy.\footnote{Under the assumption
that the insurance company has perfect foresight about agents' behavior and it is able to instantaneously adjust the policy premium to its equilibrium value, the model in \cite{Anea-19} becomes three-dimensional, as discussed in Section 5 therein.}\\
In more detail, in \cite{Anea-16} physicians are randomly paired with patients
and provide them a risky medical treatment; if an adverse event occurs, patients can choose whether or not to suit their physician
for medical malpractice, while physicians can decide whether or not to practice 
defensive medicine. In the former case physicians perform unnecessary
diagnostic and therapeutic interventions, possibly harmful to patients, in order to prevent malpractice charges. 
In symbols, Antoci et al. in \cite{Anea-16} obtain the following system:
\begin{equation}\label{16}
\left\{
\begin{array}{ll}
d\,'=d(1-d)\left(p\,\ell(E_{ND}-E_D)-C_D+C_{ND}\right)\\
\vspace{-2mm}\\
\ell\,'=\ell(1-\ell)p\left(d(E_D-E_{ND})+E_{ND}-C_L\right)
\end{array}
\right. 
\end{equation}
where $d(t)\in[0,1]$ represents the share of physicians
playing strategy $D$ (practice defensive medicine) and $\ell(t)\in[0,1]$ is the share of patients
playing strategy $L$ (litigate in case that an adverse event occurs) at time t, and $d'(t),\,\ell'(t)$ are the corresponding time derivatives. In this manner $1-d(t)\in[0,1]$ represents the share of physicians
playing strategy $ND$ (not practice defensive medicine) and $1-\ell(t)\in[0,1]$ is the share of patients
playing strategy $NL$ (not litigate in case of an adverse event) at time t. The probability that an adverse event occurs is described by $p\in (0,1)$ and it is not affected by the choice of the physician to practice or not defensive medicine.
In the case of litigation, if the physician practiced defensive medicine, the patient wins with probability $q_D\in(0,1)$, while the physician wins with probability $1-q_D;$ if instead the physician did not practice defensive medicine, the patient wins with probability
$q_{ND}\in(0,1),$ while the physician wins with probability $1-q_{ND}.$ It is assumed that $q_D<q_{ND},$ i.e., defensive medicine decreases the probability for physicians of losing an eventual litigation. Calling $R>0$ the damage suffered by a patient in case that an adverse event occurs, coinciding with the compensation he/she receives from the physician if winning the litigation, and $K>0$ the sum that the patient, if losing the litigation, pays to the physician as reparation for the legal and reputation losses, it holds that
$E_D=q_D R -(1-q_D)K$ and $E_{ND}=q_{ND} R -(1-q_{ND})K$ are the expected settlement of the litigation when the physician practiced
defensive medicine or not, respectively, with $E_{ND}>E_D;$ practicing defensive medicine costs the physician $C_D,$ while not practicing defensive medicine costs the physician $C_{ND},$ with $C_D>C_{ND}\ge 0;$ finally, $C_L>0$ is the cost faced by the patient to sue the physician for medical malpractice.\\
As concerns \cite{Anea-18}, the system obtained therein is given by:
\begin{equation}\label{18}
\left\{
\begin{array}{ll}
d\,'=d(1-d)\left(B^{PH}-P\,E\,\ell\right)\\
\vspace{-2mm}\\
\ell\,'=\ell(1-\ell)\left(-q_{ND}(C_L-p_{ND}E)+\left((q_{ND}-q_{D})C_L+P\,E\right)d\right)
\end{array}
\right. 
\end{equation}
where $d(t),\,\ell(t)\in[0,1]$ and $d'(t),\,\ell'(t)$ have a similar meaning with respect to \eqref{16}. However, this time playing strategy $D$ (practice defensive medicine) means for the physician to opt for an inferior but safer treatment, while the physician plays strategy $ND$ when he/she chooses to provide the superior, but riskier, treatment to the patient. Treatments $D$ and $ND$ produce, in addition to sure benefits, an uncertain harm to the patient, which can occur with exogenous probabilities $q_D$ and $q_{ND},$ where $0<q_D<q_{ND}<1;$ the patient, when suffering the harm, can decide to sue the physician for medical malpractice, at a cost $C_L>0;$ parameters  
$p_{D},\,p_{ND}\in[0,1]$ describe the exogenous probabilities with which the court will order the physician to pay compensation for having provided the inferior, but safer, treatment or the
superior, but riskier, treatment to the patient, respectively, with $0\le p_{ND}\le p_D\le 1.$  
 Parameter $P:=p_Dq_D-p_{ND}q_{ND}$ represents the increase (or decrease, if negative) in the ex ante probability of being condemned for a physician providing the inferior treatment $D$ to a litigious patient; $E>0$ is the compensation that the patient obtains, if winning the lawsuit, from the losing physician; $B^{PH}$ represents the physician's additional immediate benefit (or cost, if negative) of practicing defensive medicine.

Both Systems \eqref{16} and \eqref{18} are encompassed in the following general formulation
\begin{equation}\label{zan}
\left\{
\begin{array}{ll}
u\,'=u(1-u)(\sigma v-\chi)\\
\vspace{-2mm}\\
v\,'=-v(1-v)(\xi u-\phi)
\end{array}
\right. 
\end{equation}
with $u(t),\,v(t)\in[0,1]$ and $\sigma,\,\chi,\,\xi,\,\phi\in\mathbb R.$
In particular, according to \cite{HoSi-88}, if $0<\chi<\sigma$ and $0<\phi<\xi$ all orbits are closed and periodic, surrounding the unique internal equilibrium $S=\left(\frac{\phi}{\xi},\frac{\chi}{\sigma}\right),$ which is a center. In view of finding a connection with the LTMs framework introduced in Section \ref{sec-2}, in what follows we will focus just on the latter 
scenario, which also describes a battle of the sexes game with two players and two strategies, in which no one of the two dominates the other. We stress that, due to the conditions imposed on the sign of the parameters in \cite{Anea-16,Anea-18}, it holds that \eqref{16} corresponds to \eqref{zan} when identifying $d$ with $u$ and $\ell$ with $v,$ while \eqref{18} corresponds to \eqref{zan} when identifying $d$ with $v$ and $\ell$ with $u.$ As we shall better see in Subsections \ref{31} and \ref{32}, this difference between \eqref{16} and \eqref{18} implies that orbits are run clockwise in the former framework and counterclockwise in the latter.\\
In order to use some results on the orbits obtained in \cite{Maea-16}, we need to check that System \eqref{zan} fulfills the conditions on page 778 therein. Namely, setting $f(u)=1-u,\,\psi(v)=\sigma v-\chi,\,g(v)=1-v,\,\varphi(u)=\xi u-\phi,$ it holds that 
$f,\,g:(0,1)\to (0,+\infty)$ are continuous functions and that $\varphi,\,\psi:(0,1)\to\mathbb R$ are $\mathcal C^1$ maps with positive derivative on $(0,1),$ satisfying
$$
\begin{array}{ll}
\lim_{u\to 0^{+}}\varphi(u)=-\phi<0,\,\,& \lim_{v\to 0^{+}}\psi(v)=-\chi<0,\\
\vspace{-2mm}\\
\lim_{u\to 1^{-}}\varphi(u)=\xi-\phi>0,\,\,& \lim_{v\to 1^{-}}\psi(v)=\sigma-\chi>0
\end{array}
$$
under the assumption that the system admits a center in $S.$ Moreover, setting $A(u)=\int\frac{\varphi(u)}{uf(u)}\,du=\int\frac{\xi u-\phi}{u(1-u)}\,du$ and 
$B(v)=\int\frac{\psi(v)}{vg(v)}\,dv=\int\frac{\sigma v-\chi}{v(1-v)}\,dv,$ it holds that
$$\lim_{u\to 0^{+}}A(u)=\lim_{u\to 1^{-}}A(u)=\lim_{v\to 0^{+}}B(v)=\lim_{v\to 1^{-}}B(v)=+\infty.$$
Indeed, $A(u)=-\phi\log(u)+(\phi-\xi)\log(1-u)+k_1$ and $B(v)=-\chi\log(v)+(\chi-\sigma)\log(1-v)+k_2,$ with $k_1,\,k_2\in\mathbb R.$ 
Hence, System \eqref{zan} admits $H(u,v)=A(u)+B(v)$ as first integral having $S$ as minimum point and, according to the results obtained in \cite{Maea-16}, 
all its solutions are periodic and describe closed orbits contained in the (open) unit square. Due to the wide class of Hamiltonian systems considered in \cite{Maea-16}, the authors do not provide a proof of the monotonicity of the period of the orbits, but show some results about the period of small and large cycles.\footnote{Namely, Madotto et al. in \cite{Maea-16} prove, on the one hand, that the approximation of the period length of the small cycles by means of the period of the linearized system is valid near the equilibrium point and, on the other hand, that the period length of large cycles, approaching the boundary of the feasible set, is arbitrarily high, in the case that $f$ and $g$ are for instance $\mathcal C^1$ functions on the open interval $(0,1),$ that are continuous in $0,$ too, like it happens in our framework.} Nonetheless, the proof of the increasing monotonicity of the period of the orbits with the energy level for System \eqref{zan} can be found in \cite{Sc-85,Sc-90} by Renate Schaaf (see $(4)$ and the corresponding comments on page 97 in \cite{Sc-85} and Example 2.4.3 on page 64 in \cite{Sc-90}). We will rely on such result in our application of the method of the LTMs, summarized in Section \ref{sec-2}, to System \eqref{zan}, and more precisely to \eqref{16} and \eqref{18}.\\
However, in order to apply the LTMs technique, even before checking the twist condition on the boundary of the considered linked annuli as a consequence of the nonisochronicity of the centers, we need to guarantee that linked together annuli do exist. To such aim, since annuli are for us composed of energy levels, we may assume that at least one of the model parameters varies in a periodic fashion, e.g. due to a seasonal effect, so that we obtain two conservative systems, the original and the perturbed ones, and for each of them we can consider an annulus: under suitable conditions on the orbits, the two annuli cross in two or more generalized rectangles, thus resulting linked together.
In particular, we will make the periodic variation assumption on parameter $p$ in System \eqref{16} and on $q_{ND}$ in System \eqref{18}.
Namely, we recall that $p\in (0,1)$ describes the probability that an adverse event for the patient occurs in the positive defensive medicine framework, and it is not affected by the choice of the physician to practice or not defensive medicine; assuming instead negative defensive medicine, $q_{ND}$ is the probability that a harm occurs when the physician chooses to provide the superior, but riskier, treatment to the patient. Two facts are well-known and largely documented in the medical literature: that hypertension phenomena raise in winter and fall in summer in countries both north and south of the equator (see \cite{Atea-08,Deea-12,Fa-13,Siea-10}) and that hypertension increases e.g. perioperative cardiac risks, as well as the incidence of perioperative major adverse cardiovascular and cerebrovascular events (cf. for instance \cite{Hoea-04,Liuea-16}). 
In the case of positive defensive medicine, we can suppose that the physician will opt for a surgical intervention any time it is needed, while in the case of negative defensive medicine we can assume that surgery coincides with the superior, but riskier, treatment ($ND$) in comparison for instance with a less effective, but safer, pharmacological therapy ($D$). Hence, due to the seasonality in hypertension and its connection with perioperative adverse events, we can assume that both $p$ in \eqref{16} and 
$q_{ND}$ in \eqref{18} alternate in a periodic fashion between a low and a high value.\\
Despite the similar meaning of the parameters $p$ and $q_{ND}$ and, consequently, despite the common rationale for assuming a periodic variation on them, we stress that modifying the former or the latter produces a different effect on the position of the center. Namely, as we shall see in Subsection \ref{31}, an increase in parameter $q_{ND}$ in \eqref{18} will generally affect the position of both components of the center, while in Subsection \ref{32} we will find that raising $p$ in \eqref{16} affects just the ordinate of the center. Moreover, as mentioned above, orbits are run clockwise for System \eqref{16} and counterclockwise for System \eqref{18}. Since in the proofs of our results about LTMs (cf. Theorems \ref{app18} and \ref{app16}) we need to introduce the rotation number and count the laps completed by suitable paths around the center, those computations turn out to be more intuitive when orbits are run counterclockwise. For such reason, we start focusing on the framework from \cite{Anea-18} in Subsection \ref{31}, turning to the analysis of the setting in \cite{Anea-16} in Subsection \ref{32}.

\subsection{Analysis of the framework in Antoci et al. (2018)}\label{31}

Recalling that $P:=p_Dq_D-p_{ND}q_{ND},$ when setting 
$\zeta=B^{PH},\,\eta=P\,E=(p_Dq_D-p_{ND}q_{ND})E,\,\vartheta=q_{ND}(C_L-p_{ND}E),\,\kappa=(q_{ND}-q_{D})\,C_L+P\,E=q_{ND}(C_L-p_{ND}E)+q_D(p_{D}E-C_L),$
System \eqref{18} becomes
\begin{equation}\label{18b}
\left\{
\begin{array}{ll}
d\,'=d(1-d)\left(\zeta-\eta\,\ell\right)\\
\vspace{-2mm}\\
\ell\,'=\ell(1-\ell)\left(-\vartheta+\kappa d\right)
\end{array}
\right. 
\end{equation}
with $d(t),\,\ell(t)\in[0,1].$ Focusing on the case $0<\zeta<\eta$ and $0<\vartheta<\kappa,$ the center is given by 
$S=\left(\frac{\vartheta}{\kappa},\frac{\zeta}{\eta}\right)$ and it is immediate to see that if $q_{ND}$ increases, reaching an higher value $\widehat{q}_{ND},$ then $\zeta$ does not vary, while $\eta$ falls and $\vartheta,\,\kappa$ raise. Calling $\widehat\eta,\,\widehat\vartheta,\,\widehat\kappa$ the new parameter values and $V=\left(\frac{\widehat\vartheta}{\widehat\kappa},\frac{\zeta}{\widehat\eta}\right)$ the new equilibrium, it is a global center of the perturbed system
\begin{equation}\label{18p}
\left\{
\begin{array}{ll}
d\,'=d(1-d)\left(\zeta-\widehat\eta\,\ell\right)\\
\vspace{-2mm}\\
\ell\,'=\ell(1-\ell)\left(-\widehat\vartheta+\widehat\kappa d\right)
\end{array}
\right. 
\end{equation}
as long as $0<\zeta<\widehat\eta$ and $0<\widehat\vartheta<\widehat\kappa.$
In particular, due to the increase in $q_{ND},$ the ordinate of $V$ will surely be larger than that of $S,$ and a simple computation shows that the abscissa of $V$ is always larger than that of $S,$ too. Namely, this is the case illustrated in
Figures \ref{la}--\ref{sapf}, which should help the reader to better understand Definition \ref{li} of linked annuli and to follow the proof of Theorem \ref{app18}.\\
Let us start by precisely defining what two linked annuli are in the present framework.
To such aim we notice that, by the results recalled after System \eqref{zan}, the orbits for System \eqref{18b} surrounding $S$ have equation
\begin{equation}\label{hs-18}
H_S(d,\ell)=-\vartheta\log(d)+(\vartheta-\kappa)\log(1-d)-\zeta\log(\ell)+(\zeta-\eta)\log(1-\ell)=e,
\end{equation} 
for some $e\ge e_0,$ where $e_0:=-\vartheta\log(\frac{\vartheta}{\kappa})+(\vartheta-\kappa)\log(1-\frac{\vartheta}{\kappa})-\zeta\log(\frac{\zeta}{\eta})+(\zeta-\eta)\log(1-\frac{\zeta}{\eta})$ is the minimum ``energy'' level attained by $H_S(d,\ell)$
on the unit square. Indeed, by the previously recalled results from \cite{Maea-16}, $H_S$ admits $S$ as minimum point.
Setting $\Gamma_S(e)=\{(d,\ell)\in (0,1)^2:H_S(d,\ell)=e\},$ for some $e>e_0,$ this is a simple closed curve surrounding $S.$ We call {\it annulus around} $S$ any set $\mathcal C_S(e_1,e_2)=\{(d,\ell)\in (0,1)^2:e_1\le H_S(d,\ell)\le e_2\}$ with $e_0<e_1<e_2,$ so that the inner boundary of 
$\mathcal C_S(e_1,e_2)$ coincides with $\Gamma_S(e_1)$ and the outer boundary coincides with $\Gamma_S(e_2).$\\
Similarly, the orbits for System \eqref{18p} surrounding $V$ have equation
\begin{equation}\label{hv-18}
H_V(d,\ell)=-\widehat\vartheta\log(d)+(\widehat\vartheta-\widehat\kappa)\log(1-d)-\zeta\log(\ell)+(\zeta-\widehat\eta)\log(1-\ell)=h,
\end{equation}
for some $h\ge h_0,$ where $h_0:=H_V(V)$ is the minimum ``energy'' level attained by $H_V(d,\ell)$
on the unit square. The sets $\Gamma_V(h)=\{(d,\ell)\in (0,1)^2:H_V(d,\ell)=h\},$ are simple closed curves surrounding $V$ for any $h>h_0,$
and we call {\it annulus around} $V$ any set $\mathcal C_V(h_1,h_2)=\{(d,\ell)\in (0,1)^2:h_1\le H_V(d,\ell)\le h_2\}$ with 
$h_0<h_1<h_2,$ whose inner boundary coincides with $\Gamma_V(h_1)$ and whose outer boundary coincides with $\Gamma_V(h_2).$\\
Let us consider the straight line $r$ joining $S$ and $V$ and define on it the ordering inherited from the horizontal axis, so that given $P=(d_P,\ell_P)$ and $Q=(d_Q,\ell_Q)$ belonging to $r$ it holds that $P\vartriangleleft\, Q$ (resp. $P\trianglelefteq\, Q$) if and only if $d_P<d_Q$ (resp. $d_P\le d_Q$). We are now in position to introduce the following:
\begin{definition}\label{li}
Given the annulus $\mathcal C_S(e_1,e_2)$ around $S$ and the annulus $\mathcal C_V(h_1,h_2)$ around $V,$ we say that they are linked together if
$$S_{2,-}\vartriangleleft\, S_{1,-}\trianglelefteq\, V_{2,-}\vartriangleleft\, V_{1,-}\trianglelefteq\, S_{1,+}\vartriangleleft\, S_{2,+}\trianglelefteq\, V_{1,+}\vartriangleleft\, V_{2,+}$$
where, for $i\in\{1,2\},$ $S_{i,-}$ and $S_{i,+}$ denote the intersection points\footnote{We stress that, for $e_i>e_0$ and $h_i>h_0,\,i\in\{1,2\},$ the boundary sets $\Gamma_S(e_i)$ and $\Gamma_V(h_i)$ intersect the straight line $r$ in exactly two points because 
$\{(d,\ell)\in (0,1)^2:H_S(d,\ell)\le e\}$ and $\{(d,\ell)\in (0,1)^2:H_V(d,\ell)\le h\},$ coinciding with the lower contour sets of the convex functions $H_S$ in \eqref{hs-18} and $H_V$ in \eqref{hv-18}, are star-shaped for all $e>e_0$ and for every $h>h_0,$ respectively. We will need the star-shapedness of those lower contour sets along the proof of Theorem \ref{app18}, too.} 
between $\Gamma_S(e_i)$ and the straight line $r,$ with 
$S_{i,-}\vartriangleleft\, S\vartriangleleft\, S_{i,+},$ and, similarly, $V_{i,-}$ and $V_{i,+}$ denote the intersection points between $\Gamma_V(h_i)$ and $r,$ with $V_{i,-}\vartriangleleft\, V\vartriangleleft\, V_{i,+}.$
\end{definition}
We refer the reader to Figure \ref{la} for a pictorial illustration of Definition \ref{li}. In order not to overburden the pictures, we drew the unit square, containing all orbits for both Systems \eqref{18b} and \eqref{18p}, just in Figure \ref{la} (A), but we will omit to represent it in the next pictures.

\begin{figure}[ht]
\centering
\begin{minipage}[t]{0.48\textwidth}
\includegraphics[width=\textwidth,height=6cm]{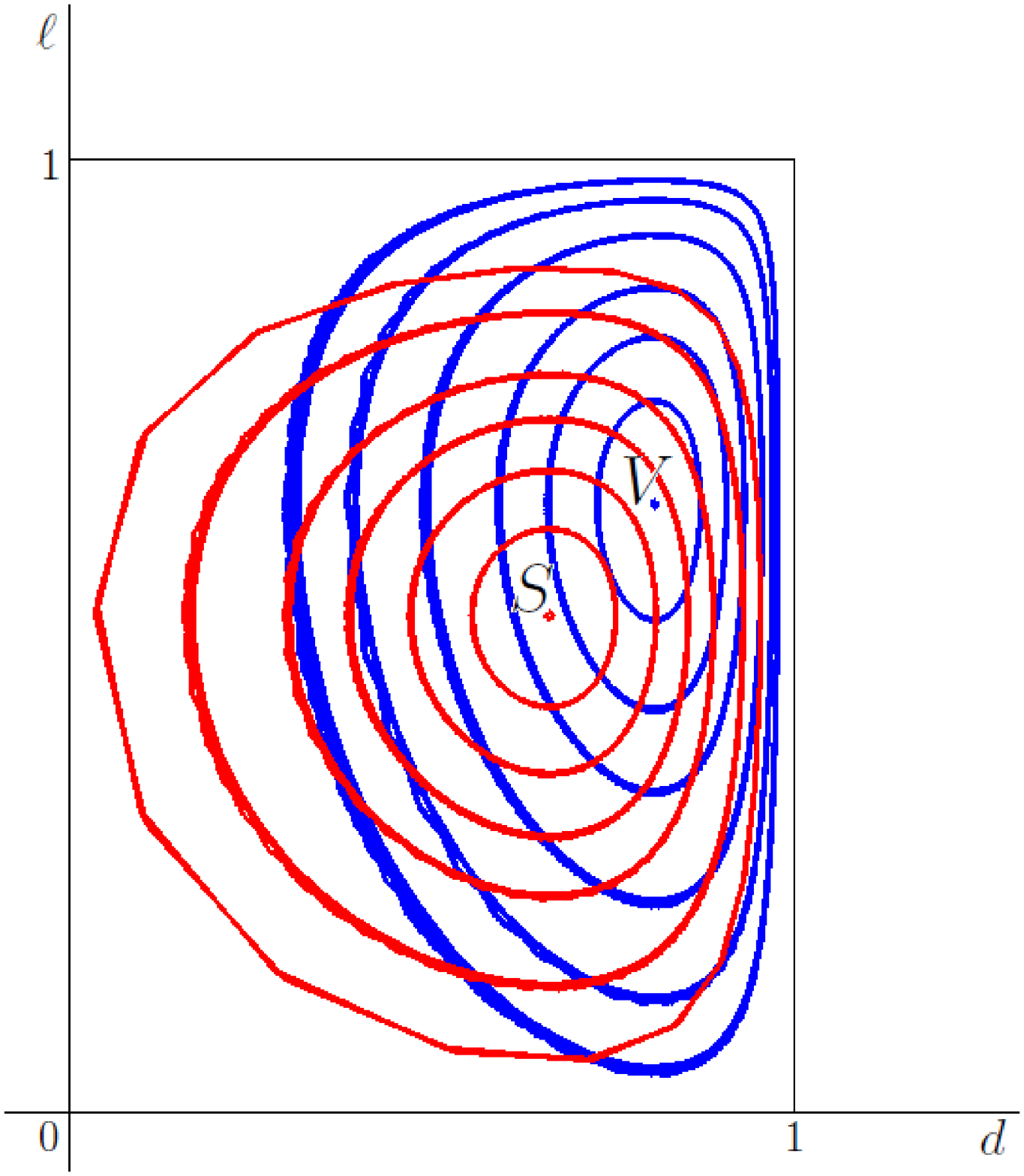}
\center{(A)}
\end{minipage}
\begin{minipage}[t]{0.48\textwidth}
\includegraphics[width=\textwidth,height=6cm]{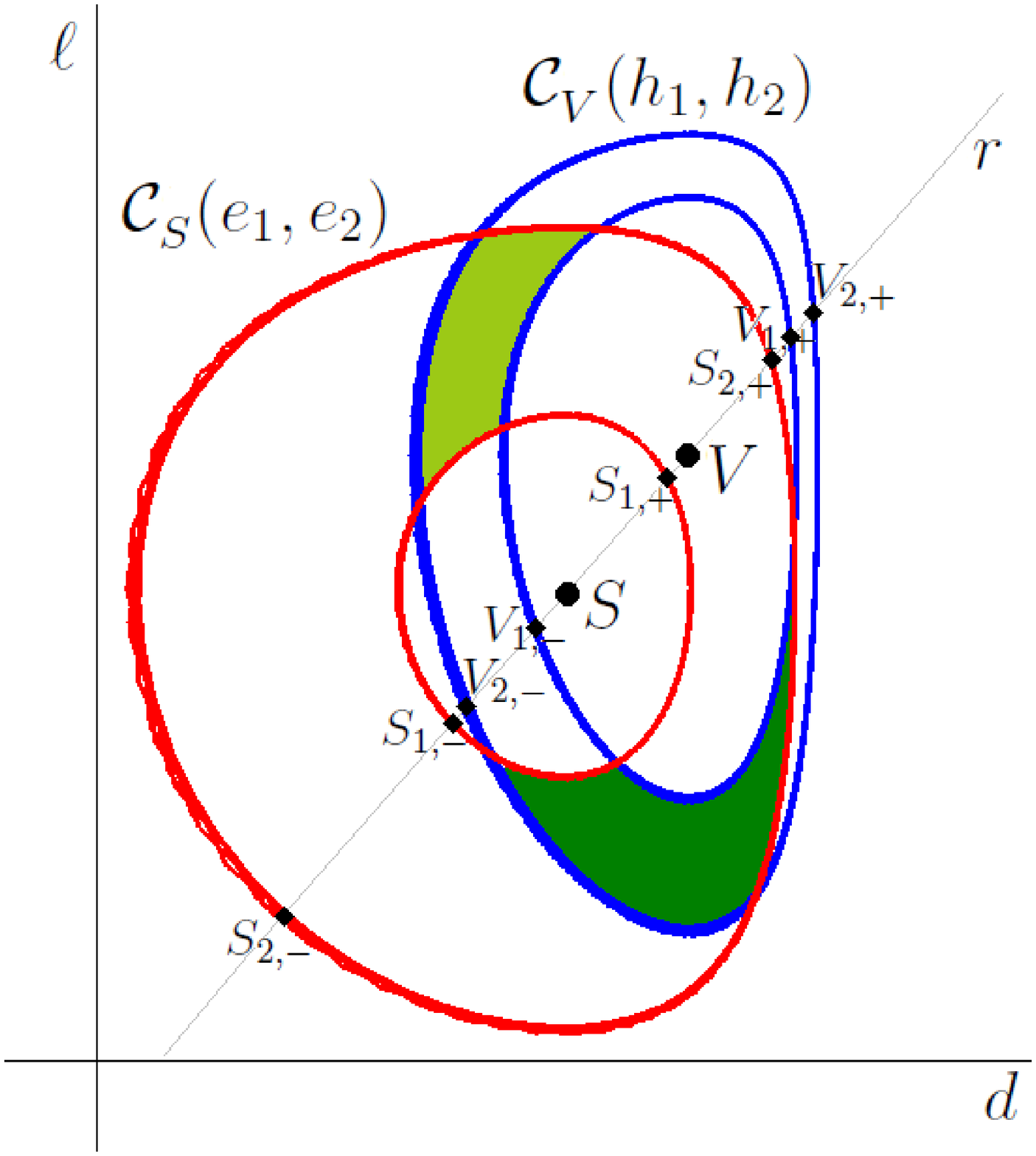}
\center{(B)}
\end{minipage}
\caption{In (A) we represent some energy level lines associated with System \eqref{18b}, in red, surrounding $S,$ as well as some energy level lines associated with System \eqref{18p}, in blue, surrounding $V.$ In (B), suitably choosing two level lines for each system, 
we obtain two linked together annuli, according to Definition \ref{li}.}\label{la}
\end{figure}

\noindent
In view of stating and proving Theorem \ref{app18}, we have to explain how the shift between the regimes described by Systems \eqref{18b} and \eqref{18p} occurs. Due to the above recalled seasonality in hypertension and its connection with perioperative adverse events, we can suppose that $q_{ND}$ in \eqref{18} alternate in a periodic fashion between a low and a high value.
We will then assume that our dynamical system is described by \eqref{18b} for $t\in[0,T_S)$, while it is described by \eqref{18p} for $t\in[T_S,T_S+T_V),$
and that the same alternation between the two regimes occurs with $T$-periodicity, where $T=T_S+T_V.$ In other terms, we can suppose that we are dealing with a system with periodic coefficients of the form
\begin{equation}\label{18t}
\left\{
\begin{array}{ll}
d\,'=d(1-d)\left(\zeta(t)-\eta(t)\ell\right)\\
\vspace{-2mm}\\
\ell\,'=\ell(1-\ell)\left(-\vartheta(t)+\kappa(t) d\right)
\end{array}
\right. 
\end{equation}
where $\zeta(t)\equiv\zeta$ and, for $x\in\{\eta,\vartheta,\kappa\},$ it holds that
\begin{equation}\label{xt}
x(t)=\left\{
\begin{array}{ll}
x \quad & \mbox{for } \, t\in[0,T_S) \\
\vspace{-2mm}\\
\widehat x \quad & \mbox {for } \, t\in[T_S,T)
\end{array}
\right.
\end{equation} 
with $0<\zeta<\widehat\eta,\,0<\vartheta<\kappa$ and $0<\widehat\vartheta<\widehat\kappa.$ 
Hence, the functions $\eta(t),\,\vartheta(t)$ and $\kappa(t)$ are 
piecewise constant and they are supposed to be extended to
the whole real line by $T$-periodicity.\\
Let us finally introduce our last needed ingredient, i.e., the Poincar\'e map $\Psi$ of System \eqref{18t}, which associates with any initial condition $(d_0,\ell_0)$ belonging to the open unit square the position at time $T$ of the solution 
$\varsigma(\,\cdot\,,(d_0,\ell_0))=(d(\,\cdot\,,(d_0,\ell_0)),\ell(\,\cdot\,,(d_0,\ell_0)))$ to \eqref{18t} starting at time $t=0$ from $(d_0,\ell_0).$
In symbols, $\Psi:(0,1)^2\to(0,1)^2,\,(d_0,\ell_0)\mapsto\varsigma(T,(d_0,\ell_0)).$
We remark that, along the paper, solutions are meant in the Carath\'eodory sense, i.e., they are absolutely continuous and satisfy the corresponding system for almost every $t\in\mathbb R.$ Notice that 
$\Psi$ may be decomposed as $\Psi=\Psi_V\,\circ\Psi_S,$ where $\Psi_S$ is the Poincar\'e map associated with System \eqref{18b} for $t\in[0,T_S]$ and $\Psi_V$ is the Poincar\'e map associated with System \eqref{18p} for $t\in[0,T_V].$ Moreover, since every annulus $\mathcal C_S(e_1,e_2)$ around $S$ is invariant under the action of the map $\Psi_S,$ being composed of the invariant orbits $\Gamma_S(e),$ for $e\in[e_1,e_2],$ and, similarly, since every annulus $\mathcal C_V(h_1,h_2)$ around $V$ is invariant under the action of the map $\Psi_V,$ it holds that every pair of linked twist annuli is invariant under the action of the composite map $\Psi.$
In Theorem \ref{app18} we will denote by $\tau_S(e),$ for all $e>e_0,$ the period of $\Gamma_S(e),$ i.e., the time needed by the solution 
$\varsigma_S(\,\cdot\,,(d_0,\ell_0))$ to System \eqref{18b}, starting from any $(d_0,\ell_0)\in\Gamma_S(e),$ to complete one turn around $S$ moving along $\Gamma_S(e),$ and by $\tau_V(h),$ for all $h>h_0,$ the period of $\Gamma_V(h),$ i.e., the time needed by the solution 
$\varsigma_V(\,\cdot\,,(d_0,\ell_0))$ to System \eqref{18p}, starting from any $(d_0,\ell_0)\in\Gamma_V(h),$ to complete one turn around $V$ moving along $\Gamma_V(h).$ A straightforward analysis of the phase portrait shows that both orbits surrounding $S$ or $V$ are run counterclockwise. As recalled above, the increasing monotonicity of $\tau_S(\,\cdot\,)$ and $\tau_V(\,\cdot\,)$ with the energy levels has been proven in \cite{Sc-85,Sc-90}. Hence, for any annulus $\mathcal C_S(e_1,e_2)$ around $S$ it holds that $\tau_S(e_1)<\tau_S(e_2),$ as well as for each annulus $\mathcal C_V(h_1,h_2)$ around $V$ it holds that $\tau_V(h_1)<\tau_V(h_2).$\\
Our result about System \eqref{18t} reads as follows:
\begin{theorem}\label{app18}
For any choice of the parameters $0<\zeta<\eta$ and $0<\vartheta<\kappa,$ defined like in System \eqref{18b}, and for any increase in the parameter $q_{ND}$ such that $0<\zeta<\widehat\eta$ and $0<\widehat\vartheta<\widehat\kappa,$ given the annulus $\mathcal C_S(e_1,e_2)$ around $S,$ for some $e_0<e_1<e_2,$ and the annulus $\mathcal C_V(h_1,h_2)$ around $V,$ for some $h_0<h_1<h_2,$ assume that they are linked together, calling $\mathcal A$ and $\mathcal B$ the connected components of $\mathcal C_S(e_1,e_2)\cap\mathcal C_V(h_1,h_2).$
Then, if $T_S>\frac{11\,\tau_S(e_1)\tau_S(e_2)}{2\,(\tau_S(e_2)-\tau_S(e_1))}$ and 
$T_V>\frac{9\,\tau_V(h_1)\tau_V(h_2)}{2\,(\tau_V(h_2)-\tau_V(h_1))},$ the 
Poincar\'e map $\Psi=\Psi_V\circ\Psi_S$ of System \eqref{18t} induces chaotic dynamics on two symbols in $\mathcal A$ and in $\mathcal B,$ and thus all the properties listed in Theorem \ref{th} are fulfilled for $\Psi$.
\end{theorem}
We can summarize the statement of Theorem \ref{app18} by saying that any time we have two linked together annuli corresponding to Systems \eqref{18b} and \eqref{18p}, then, if the switching times between the regimes described by Systems \eqref{18b} and \eqref{18p} are sufficiently large, the Poincar\'e map $\Psi=\Psi_V\circ\Psi_S$ induces chaotic dynamics on two symbols in the sets in which the two annuli cross, and thus all properties listed in Theorem \ref{th} hold for $\Psi.$

\smallskip
\noindent {\textit{Proof of Theorem \ref{app18}.}}
Given the linked together annuli $\mathcal C_S(e_1,e_2)$ and $\mathcal C_V(h_1,h_2),$ let us set $\mathcal C_S(e_1,e_2)=\mathcal C_S^u(e_1,e_2)\cup\mathcal C_S^d(e_1,e_2)$ and $\mathcal C_V(h_1,h_2)=\mathcal C_V^u(h_1,h_2)\cup\mathcal C_V^d(h_1,h_2),$ where $u$ (resp. $d$) stands for ``up'' (resp. ``down''), and indeed $\mathcal C_S^u(e_1,e_2)$ (resp. $\mathcal C_S^d(e_1,e_2)$) is the subset of $\mathcal C_S(e_1,e_2)$ which lies above (resp. below) the straight line $r,$ joining $S$ and $V,$ and, analogously, $\mathcal C_V^u(h_1,h_2)$ (resp. $\mathcal C_V^d(h_1,h_2)$) is the subset of $\mathcal C_V(h_1,h_2)$ which lies above (resp. below) the straight line $r.$ See Figure \ref{ud} for a graphical illustration.

\begin{figure}[ht]
\centering
\includegraphics[width=7cm,height=6cm]{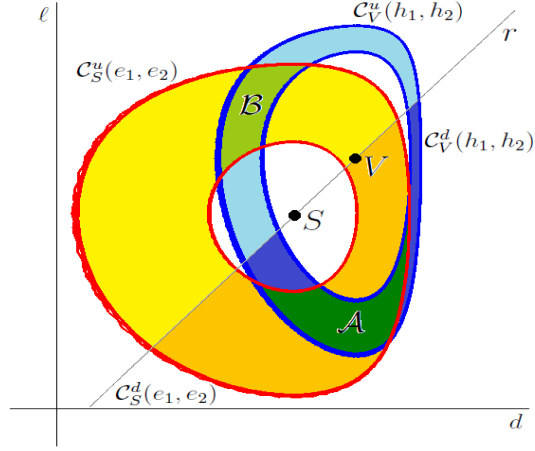}
\caption{Given the linked together annuli $\mathcal C_S(e_1,e_2)$ and $\mathcal C_V(h_1,h_2),$ we represent in yellow the set $\mathcal C_S^u(e_1,e_2)$ and in light blue the set $\mathcal C_V^u(h_1,h_2),$ both lying above the straight line $r,$ joining the centers $S$ and $V,$ while we represent in orange the set $\mathcal C_S^d(e_1,e_2)$ and in dark blue the set $\mathcal C_V^d(h_1,h_2),$ both lying below the straight line $r.$ Notice that $\mathcal C_S^d(e_1,e_2)$ and $\mathcal C_V^d(h_1,h_2)$ cross in $\mathcal A$ (colored in dark green), and that 
$\mathcal C_S^u(e_1,e_2)$ and $\mathcal C_V^u(h_1,h_2)$ cross in $\mathcal B$ (colored in light green).}\label{ud}
\end{figure}

\noindent
Let us also set $\mathcal A:=\mathcal C_S^d(e_1,e_2)\cap \mathcal C_V^d(h_1,h_2)$ and $\mathcal B:=\mathcal C_S^u(e_1,e_2)\cap\mathcal C_V^u(h_1,h_2).$ We are going to show that, when orienting them by setting ${\mathcal A}^-=\mathcal A^{-}_{l}\cup\mathcal A^{-}_{r}$ and ${\mathcal B}^-=\mathcal B^{-}_{l}\cup\mathcal B^{-}_{r},$ with $\mathcal A^{-}_{l}:=\mathcal A\cap\Gamma_S(e_1),\,\mathcal A^{-}_{r}:=\mathcal A\cap\Gamma_S(e_2),\,\mathcal B^{-}_{l}:=\mathcal B\cap\Gamma_V(h_1),\,\mathcal B^{-}_{r}:=\mathcal B\cap\Gamma_V(h_2),$ then conditions $(C_F)$ and $(C_G)$ in Theorem \ref{th} are fulfilled for the oriented rectangles ${\widetilde{\mathcal A}} := ({\mathcal A},{\mathcal A}^-)$ and
${\widetilde{\mathcal B}} := ({\mathcal B},{\mathcal B}^-)$ with $F=\Psi_S$ and $G=\Psi_V.$ Then the 
Poincar\'e map $\Psi=\Psi_V\circ\Psi_S$ of System \eqref{18t} induces chaotic dynamics on two symbols in $\mathcal A$ 
and thus, since $\Psi$ is a homeomorphism on the open unit square, 
all the properties listed in Theorem \ref{th} are fulfilled for $\Psi$. This will conclude the verification of the first half of our result. Namely, the proof will be complete when we will show that conditions $(C_F)$ and $(C_G)$ in Theorem \ref{th} are fulfilled with $F=\Psi_S$ and $G=\Psi_V$ also when interchanging the role of $\mathcal A$ and $\mathcal B,$ after having suitably modified their orientation, since from this it will follow that $\Psi$ induces chaotic dynamics on two symbols in $\mathcal B,$ too.\\
Checking that condition $(C_F)$ in Theorem \ref{th} is fulfilled with $F=\Psi_S$ with respect to ${\widetilde{\mathcal A}}$ and ${\widetilde{\mathcal B}}$ amounts to find two disjoint compact subsets ${\mathcal H_0},\,{\mathcal H_1}$ of ${\mathcal A}$
such that $\displaystyle{({\mathcal H}_i,\Psi_S): {\widetilde{\mathcal A}}\stretchx\, {\widetilde{\mathcal B}}},$ for $i=0,1.$ 
To such aim, we need to consider the image $\overline\gamma$ through $\Psi_S$ of any path $\gamma:[0,1]\to{\mathcal A}$ joining $\mathcal A^{-}_{l}$ with 
$\mathcal A^{-}_{r}$ and, recalling that orbits for System \eqref{18b} are run counterclockwise, count how many times it completely crosses 
${\mathcal B},$ from $\mathcal B^{-}_{l}$ to $\mathcal B^{-}_{r},$ when $T_S$ is large enough. In particular, we will show that, in order for ${\mathcal H_0},\,{\mathcal H_1}$ as above to exist, two complete crossings have to occur between $\overline\gamma$ and ${\mathcal B}.$ This will allow us to determine as lower bound on $T_S$ the constant $k_S:=\frac{11\,\tau_S(e_1)\tau_S(e_2)}{2\,(\tau_S(e_2)-\tau_S(e_1))},$ recalling that $\tau_S(e_1)<\tau_S(e_2).$\\
In view of counting the turns completed by the image of a path around $S,$ we introduce a system of generalized polar coordinates centered at $S.$ Namely, assuming to have performed the rototranslation of $\mathbb R^2$ that brings the origin
to the point $S$ and makes the horizontal axis coincide with the straight line $r,$ we can express the solution 
$\varsigma_S(t\,,(d_0,\ell_0))=(d(t\,,(d_0,\ell_0)),\ell(t\,,(d_0,\ell_0)))$ to System \eqref{18b} with initial point in
$(d_0,\ell_0)\in (0,1)^2$ through the radial coordinate
$\rho_S(t,(d_0,\ell_0))=\|\varsigma_S(t\,,(d_0,\ell_0))-S\|$ and the angular coordinate $\theta_S(t,(d_0,\ell_0))$ as 
$\varsigma_S(t\,,(d_0,\ell_0))=\rho_S(t,(d_0,\ell_0))\,(\cos(\theta_S(t,(d_0,\ell_0))),\,\sin(\theta_S(t,(d_0,\ell_0)))).$ Since orbits for System \eqref{18b} are run counterclockwise, we define 
the rotation number, describing the normalized angular displacement of $\varsigma_S(t\,,(d_0,\ell_0))$ during the time interval $[0,t]\subseteq [0,T_S],$ as 
\begin{equation}\label{rns}
\rot_S(t,(d_0,\ell_0)):=\frac{\theta_S(t,(d_0,\ell_0))-\theta_S(0,(d_0,\ell_0))}{2\pi}\,.
\end{equation}
For $e>e_0,$ as a consequence of the star-shapedness with respect to $S$ of the lower contour sets $\{(d,\ell)\in (0,1)^2:H_S(d,\ell)\le e\},$
 with $H_S$ as in \eqref{hs-18},
and recalling that $\tau_S(e)$ is the time needed by the solution $\varsigma_S(\,\cdot\,,(d_0,\ell_0))$ to System \eqref{18b}, starting from any $(d_0,\ell_0)\in\Gamma_S(e),$ to complete one turn around $S$ moving along $\Gamma_S(e),$
the following properties hold true for every $(d_0,\ell_0)\in\Gamma_S(e),$ for all $t\in [0,T_S]$ and for each $n\in {\mathbb N}\setminus\{0\}$:
\begin{equation}\label{rot}
\begin{array}{ll}
\rot_S(t,(d_0,\ell_0)) < n \, &\Longleftrightarrow  \,\,\, t < n \, \tau_S(e)\\
\vspace{-3mm}\\
\rot_S(t,(d_0,\ell_0)) = n \, &\Longleftrightarrow  \,\,\, t = n \, \tau_S(e)\\
\vspace{-3mm}\\
\rot_S(t,(d_0,\ell_0)) > n \, &\Longleftrightarrow  \,\,\, t > n \, \tau_S(e)
\end{array}
\end{equation}
so that $\rot_S(t,(d_0,\ell_0))\in(n,n+1) \Longleftrightarrow t\in(n \,\tau_S(e),(n + 1)\, \tau_S(e)).$\\
As explained above, in order to check that condition $(C_F)$ in Theorem \ref{th} is fulfilled with $F=\Psi_S$ with respect to ${\widetilde{\mathcal A}}$ and ${\widetilde{\mathcal B}},$ we have to find two disjoint compact subsets ${\mathcal H_0},\,{\mathcal H_1}$ of ${\mathcal A}$
such that $\displaystyle{({\mathcal H}_i,\Psi_S): {\widetilde{\mathcal A}}\stretchx\, {\widetilde{\mathcal B}}},$ for $i=0,1.$ Let then 
$\gamma:[0,1]\to{\mathcal A},\,\lambda\mapsto\gamma(\lambda),$ be any path such that $\gamma(0)\in\mathcal A^{-}_{l}$ and $\gamma(1)\in\mathcal A^{-}_{r}.$ For every $\lambda\in [0,1],$ we consider the position at time $T_S$ of the solution $\varsigma_S(t\,,\gamma(\lambda))$ to System \eqref{18b} starting at $t=0$ from $\gamma(\lambda)\in{\mathcal A},$ i.e., $\Psi_S(\gamma(\lambda)),$ as well as the angular coordinate 
$\theta_S(T_S,\gamma(\lambda))$ associated with it. Notice that the function $\lambda\mapsto\theta_S(T_S,\gamma(\lambda))$ is continuous due to the continuity of $\gamma$ and by the continuous dependence of the solutions from the initial data.
Since $\mathcal A^{-}_{l}\subset\Gamma_S(e_1)$ and $\mathcal A^{-}_{r}\subset\Gamma_S(e_2),$ recalling that $\tau_S(e_1)<\tau_S(e_2),$ we claim that if $T_S\ge\frac{11\,\tau_S(e_1)\tau_S(e_2)}{2\,(\tau_S(e_2)-\tau_S(e_1))}$ then\footnote{We stress that the proof of the properties regarding the system centered at $S$ is valid also when $T_S=\frac{11\,\tau_S(e_1)\tau_S(e_2)}{2\,(\tau_S(e_2)-\tau_S(e_1))}$ and, similarly, the verification of the features concerning the system centered at $V$ works even when $T_V=\frac{9\,\tau_V(h_1)\tau_V(h_2)}{2\,(\tau_V(h_2)-\tau_V(h_1))}.$ However, the strict inequalities make our result robust with respect to small perturbations in the coefficients of System \eqref{18t}, as better explained below.} 
$\theta_S(T_S,\gamma(0))-\theta_S(T_S,\gamma(1))>5\pi.$ If this is the case, there exists $n^{\ast}\in\mathbb N$ such that 
$[2n^{\ast}\pi,2n^{\ast}\pi+\pi]$ and $[2(n^{\ast}+1)\pi,2(n^{\ast}+1)\pi+\pi]$ are contained in the interval 
$\{\theta_S(T_S,\gamma(\lambda)):\lambda\in[0,1]\}.$ Hence, by Bolzano theorem, there exist two disjoint maximal intervals 
$[\lambda_0',\lambda_0''],\,[\lambda_1',\lambda_1'']$ of $[0,1]$ such that for $i\in\{0,1\}$ it holds that
$\{\theta_S(T_S,\gamma(\lambda)):\lambda\in[\lambda_i',\lambda_i'']\}\subseteq[2(n^{\ast}+i)\pi,2(n^{\ast}+i)\pi+\pi],$
with $\theta_S(T_S,\gamma(\lambda_i'))=2(n^{\ast}+i)\pi$ and $\theta_S(T_S,\gamma(\lambda_i''))=2(n^{\ast}+i)\pi+\pi.$
 We can then set 
${\mathcal H}_i:=\{(d_0,\ell_0)\in\mathcal A:\theta_S(T_S,(d_0,\ell_0))\in[2(n^{\ast}+i)\pi,2(n^{\ast}+i)\pi+\pi]\}$ for $i\in\{0,1\}$
in order to have the stretching relation 
\begin{equation}\label{sr}
\displaystyle{({\mathcal H}_i,\Psi_S): {\widetilde{\mathcal A}}\stretchx\, {\widetilde{\mathcal B}}} 
\end{equation}
satisfied. Namely, for $i\in\{0,1\},$ ${\mathcal H}_i$ is a compact set containing $\{\gamma(\lambda)):\lambda\in[\lambda_i',\lambda_i'']\}.$ Moreover, we have proved that, for $i\in\{0,1\}$ and $\lambda\in[\lambda_i',\lambda_i''],$ it holds that 
$\gamma(\lambda)\in{\mathcal H}_i,\,\Psi_S(\gamma(\lambda))\in\mathcal C_S^u(e_1,e_2)$ and $H_S(\Psi_S(\gamma(\lambda_i')))\le h_1,$
 $H_S(\Psi_S(\gamma(\lambda_i'')))\ge h_2.$ Then, there exists $[\lambda_i^{\ast},\lambda_i^{\ast\ast}]\subseteq[\lambda_i',\lambda_i'']$ such that $\Psi_S(\gamma(\lambda))\in\mathcal B$ for every $\lambda\in[\lambda_i^{\ast},\lambda_i^{\ast\ast}],$ and 
$H_S(\Psi_S(\gamma(\lambda_i^{\ast})))=h_1,\,H_S(\Psi_S(\gamma(\lambda_i^{\ast\ast})))=h_2.$ 
Since $\mathcal B^{-}_{l}:=\mathcal B\cap\Gamma_V(h_1),\,\mathcal B^{-}_{r}:=\mathcal B\cap\Gamma_V(h_2),$ this means that $\Psi_S(\gamma(\lambda_i^{\ast}))\in\mathcal B^{-}_{l}$ and $\Psi_S(\gamma(\lambda_i^{\ast\ast}))\in\mathcal B^{-}_{r},$ concluding the verification of \eqref{sr}. See Figure \ref{sapf} (A).\\
We then have to check our claim that $T_S\ge\frac{11\,\tau_S(e_1)\tau_S(e_2)}{2\,(\tau_S(e_2)-\tau_S(e_1))}$ implies that\\ 
$\theta_S(T_S,\gamma(0))-\theta_S(T_S,\gamma(1))>5\pi,$ where $\gamma:[0,1]\to{\mathcal A}$ is any path such that $\gamma(0)\in\mathcal A^{-}_{l}:=\mathcal A\cap\Gamma_S(e_1)$ and $\gamma(1)\in\mathcal A^{-}_{r}:=\mathcal A\cap\Gamma_S(e_2).$ Since $\rot_S(t,\gamma(0)) \geq \lfloor t/\tau_S(e_1)\rfloor$ and
$\rot_S(t,\gamma(1)) \leq \lceil t/\tau_S(e_2)\rceil$ for every $t> 0,$
it follows that
\begin{equation}\label{rots}
\rot_S(t,\gamma(0)) - \rot_S(t,\gamma(1))\ge \lfloor t/\tau_S(e_1)\rfloor-\lceil t/\tau_S(e_2)\rceil > t\, \frac{\tau_S(e_2) - \tau_S(e_1)}{\tau_S(e_1)\,\tau_S(e_2)}\, - 2
\end{equation}
for every $t>0.$
As a consequence, for $T_S\ge\frac{11\,\tau_S(e_1)\tau_S(e_2)}{2\,(\tau_S(e_2)-\tau_S(e_1))}$ it holds that
$$\rot_S(T_S,\gamma(0)) - \rot_S(T_S,\gamma(1))> T_S\, \frac{\tau_S(e_2) - \tau_S(e_1)}{\tau_S(e_1)\,\tau_S(e_2)}\, - 2\ge\frac{11}{2}-2=\frac{7}{2}>3.$$
Hence, $\theta_S(T_S,\gamma(0)) - \theta_S(T_S,\gamma(1))>6\pi+\theta_S(0,\gamma(0))-\theta_S(0,\gamma(1)).$ Since 
$\overline\gamma\subset\mathcal A:=\mathcal C_S^d(e_1,e_2)\cap\mathcal C_V^d(h_1,h_2),$ we have that 
$\theta_S(0,\gamma(0)),\,\theta_S(0,\gamma(1))\in [-\pi,0],$ and thus $\theta_S(0,\gamma(0))-\theta_S(0,\gamma(1))>-\pi,$ from which it follows that
$\theta_S(T_S,\gamma(0)) - \theta_S(T_S,\gamma(1))>5\pi,$ as desired.\\
Let us now turn to the proof that condition $(C_G)$ in Theorem \ref{th} is fulfilled with $G=\Psi_V$ with respect to ${\widetilde{\mathcal B}}$ and ${\widetilde{\mathcal A}},$ i.e., that $\displaystyle{\Psi_V: {\widetilde{\mathcal B}}\stretchx\, {\widetilde{\mathcal A}}}.$ 
To such aim, we need to consider the image through $\Psi_V$ of any path $\sigma:[0,1]\to{\mathcal B}$ joining $\mathcal B^{-}_{l}$ with 
$\mathcal B^{-}_{r}$ and check that it completely crosses ${\mathcal A},$ from $\mathcal A^{-}_{l}$ to $\mathcal A^{-}_{r},$ at least once
when $T_V\ge\frac{9\,\tau_V(h_1)\tau_V(h_2)}{2\,(\tau_V(h_2)-\tau_V(h_1))},$ recalling that orbits for System \eqref{18p} are run counterclockwise and that $\tau_V(h_1)<\tau_V(h_2).$
Also in this case we introduce a system of generalized polar coordinates, but centered at $V.$ Namely, assuming to have performed the rototranslation of $\mathbb R^2$ that brings the origin
to the point $V$ and makes the horizontal axis coincide with the straight line $r,$ we can express the solution 
$\varsigma_V(t\,,(d_0,\ell_0))=(d(t\,,(d_0,\ell_0)),\ell(t\,,(d_0,\ell_0)))$ to System \eqref{18p} with initial point in
$(d_0,\ell_0)\in (0,1)^2$ through the radial coordinate
$\rho_V(t,(d_0,\ell_0))=\|\varsigma_V(t\,,(d_0,\ell_0))-V\|$ and the angular coordinate $\theta_V(t,(d_0,\ell_0))$ as 
$\varsigma_V(t\,,(d_0,\ell_0))=\rho_V(t,(d_0,\ell_0))\,(\cos(\theta_V(t,(d_0,\ell_0))),\,\sin(\theta_V(t,(d_0,\ell_0)))).$ Since orbits for System \eqref{18p} are still run counterclockwise, we can define  
the rotation number as 
\begin{equation}\label{rnv}
\rot_V(t,(d_0,\ell_0)):=\frac{\theta_V(t,(d_0,\ell_0))-\theta_V(0,(d_0,\ell_0))}{2\pi}\,,
\end{equation}
describing the normalized angular displacement during the time interval $[0,t]\subseteq [0,T_V]$ of $\varsigma_V(t\,,(d_0,\ell_0)).$\\
Let then $\sigma:[0,1]\to{\mathcal B}$ be any path such that $\sigma(0)\in\mathcal B^{-}_{l}$ and
$\sigma(1)\in\mathcal B^{-}_{r}.$ We are going to show that if $T_V\ge\frac{9\,\tau_V(h_1)\tau_V(h_2)}{2\,(\tau_V(h_2)-\tau_V(h_1))}$
then $\theta_V(T_V,\sigma(0))-\theta_V(T_V,\sigma(1))>3\pi.$
Since $\mathcal B^{-}_{l}:=\mathcal B\cap\Gamma_V(h_1),\,\mathcal B^{-}_{r}:=\mathcal B\cap\Gamma_V(h_2),$ it holds that
$\rot_V(t,\sigma(0)) \geq \lfloor t/\tau_V(h_1)\rfloor$ and
$\rot_V(t,\sigma(1)) \leq \lceil t/\tau_V(h_2)\rceil$ for every $t> 0,$
and thus
$\rot_V(t,\sigma(0)) - \rot_V(t,\sigma(1))>t\, \frac{\tau_V(h_2) - \tau_V(h_1)}{\tau_V(h_1)\,\tau_V(h_2)}\, - 2,$
similar to \eqref{rots}. Hence, for\\ 
$T_V\ge\frac{9\,\tau_V(h_1)\tau_V(h_2)}{2\,(\tau_V(h_2)-\tau_V(h_1))}$ it holds that
\begin{equation*}
\rot_V(T_V,\sigma(0)) - \rot_V(T_V,\sigma(1))> T_V\, \frac{\tau_V(h_2) - \tau_V(h_1)}{\tau_V(h_1)\,\tau_V(h_2)}\, - 2\ge\frac{9}{2}-2=\frac{5}{2}>2.
\end{equation*}
Recalling that $\mathcal B:=\mathcal C_S^u(e_1,e_2)\cap\mathcal C_V^u(h_1,h_2),$ since 
$\overline\sigma\subset\mathcal B,$ we have that \\
$\theta_V(T_V,\sigma(0)) - \theta_V(T_V,\sigma(1))>4\pi+\theta_V(0,\sigma(0))-\theta_V(0,\sigma(1))>4\pi-\pi=3\pi,$ as needed,
because $\theta_V(0,\sigma(0)),\,\theta_V(0,\sigma(1))\in [0,\pi],$ and thus $\theta_V(0,\sigma(0))-\theta_V(0,\sigma(1))>-\pi.$\\
Then, there exists $n^{\ast\ast}\in\mathbb N$ such that 
$[(2n^{\ast\ast}+1)\pi,(2n^{\ast\ast}+2)\pi]\subset\{\theta_V(T_V,\sigma(\lambda)):\lambda\in[0,1]\}$ and, consequently, there exists 
$[\lambda',\lambda'']\subseteq [0,1]$ such that \\
$\{\theta_V(T_V,\sigma(\lambda)):\lambda\in[\lambda',\lambda'']\}\subseteq[(2n^{\ast\ast}+1)\pi,(2n^{\ast\ast}+2)\pi],$
with $\theta_V(T_V,\sigma(\lambda'))=(2n^{\ast\ast}+1)\pi$ and $\theta_V(T_V,\sigma(\lambda''))=(2n^{\ast\ast}+2)\pi.$
Thus, $\Psi_V(\sigma(\lambda))\in\mathcal C_V^d(h_1,h_2)$ for every $\lambda\in[\lambda',\lambda'']$
and $H_V(\Psi_V(\sigma(\lambda')))\le e_1,\,
H_V(\Psi_V(\sigma(\lambda'')))\ge e_2.$ Then, there exists $[\lambda^{\ast},\lambda^{\ast\ast}]\subseteq[\lambda',\lambda'']$ such that $\Psi_V(\sigma(\lambda))\in\mathcal A$ for every $\lambda\in[\lambda^{\ast},\lambda^{\ast\ast}],$ and it holds that
$H_V(\Psi_V(\sigma(\lambda^{\ast})))=e_1,\,H_V(\Psi_V(\sigma(\lambda^{\ast\ast})))=e_2,$ so that  
the stretching relation 
$\displaystyle{\Psi_V: {\widetilde{\mathcal B}}\stretchx\, {\widetilde{\mathcal A}}}$ 
is verified. See Figure \ref{sapf} (B).\\
Since we have checked that conditions $(C_F)$ and $(C_G)$ in Theorem \ref{th} are fulfilled for the oriented rectangles ${\widetilde{\mathcal A}} := ({\mathcal A},{\mathcal A}^-)$ and
${\widetilde{\mathcal B}} := ({\mathcal B},{\mathcal B}^-)$ with $F=\Psi_S$ and $G=\Psi_V,$ it follows that $\Psi=\Psi_V\circ\Psi_S$
induces chaotic dynamics on two symbols in $\mathcal A.$ Given that $\Psi$ is a homeomorphism on the open unit square, it is injective on the set ${\mathcal H}:=({\mathcal H}_0\cup{\mathcal H}_1)\cap{\Psi_S}^{-1}(\mathcal B),$ and thus
all the properties listed in Theorem \ref{th} are fulfilled for $\Psi$.\\
In order to verify that $\Psi$ induces chaotic dynamics on two symbols in $\mathcal B,$ too, we orientate $\mathcal A$ by setting $\mathcal A^{-}_{l}:=\mathcal A\cap\Gamma_V(h_1),\,\mathcal A^{-}_{r}:=\mathcal A\cap\Gamma_V(h_2),\,\mathcal B^{-}_{l}:=\mathcal B\cap\Gamma_S(e_1),\,\mathcal B^{-}_{r}:=\mathcal B\cap\Gamma_S(e_2),$ and we should
show that the image through $\Psi_S$ of any path joining in $\mathcal B$ the sides $\mathcal B^{-}_{l}$ and 
$\mathcal B^{-}_{r}$ crosses ${\mathcal A},$ from $\mathcal A^{-}_{l}$ to $\mathcal A^{-}_{r},$ at least twice when 
$T_S\ge\frac{11\,\tau_S(e_1)\tau_S(e_2)}{2\,(\tau_S(e_2)-\tau_S(e_1))},$ and then check that the image through $\Psi_V$ of any path in $\mathcal A$ joining $\mathcal A^{-}_{l}$ with 
$\mathcal A^{-}_{r}$ crosses ${\mathcal B},$ from $\mathcal B^{-}_{l}$ to $\mathcal B^{-}_{r},$ at least once
when $T_V\ge\frac{9\,\tau_V(h_1)\tau_V(h_2)}{2\,(\tau_V(h_2)-\tau_V(h_1))}.$ Namely, this amounts to show that $\displaystyle{({\mathcal K}_i,\Psi_S): {\widetilde{\mathcal B}}\stretchx\, {\widetilde{\mathcal A}}}$ for suitable compact disjoint subsets ${\mathcal K}_i$ of $\mathcal B,$ for $i\in\{0,1\},$ as well as that $\displaystyle{\Psi_V: {\widetilde{\mathcal A}}\stretchx\, {\widetilde{\mathcal B}}}.$ 
Due to the similarity with the above proved properties, we leave the details in the verification of new conditions $(C_F)$ and $(C_G)$ in Theorem \ref{th} to the reader.
Then, by Theorem \ref{th} it holds that $\Psi$ induces chaotic dynamics on two symbols in $\mathcal B,$ as desired.\\
The proof is complete. 
$\hfill\square$ 

\begin{figure}[ht]
\centering
\begin{minipage}[t]{0.48\textwidth}
\includegraphics[width=\textwidth,height=6cm]{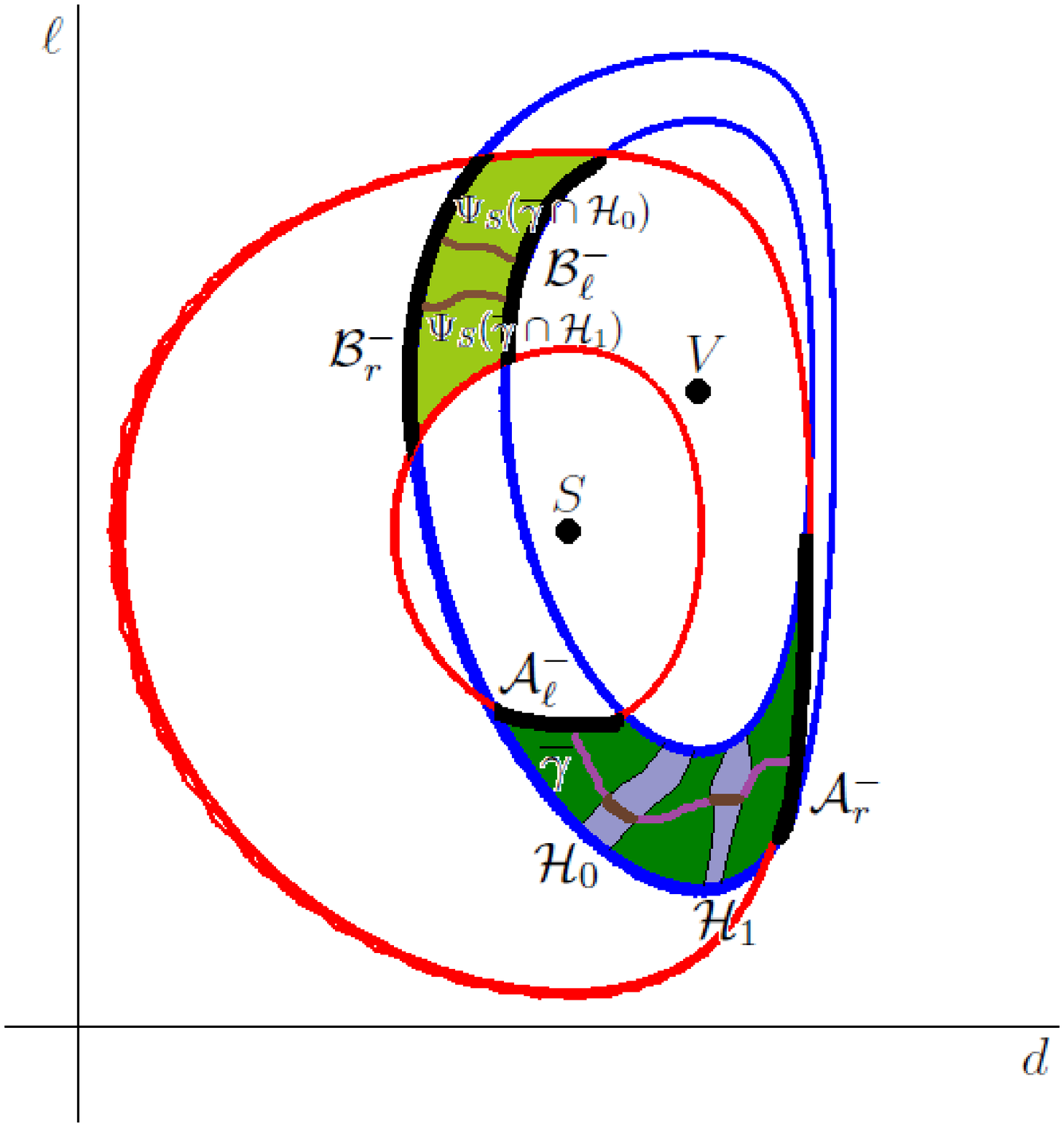}
\center{(A)}
\end{minipage}
\begin{minipage}[t]{0.48\textwidth}
\includegraphics[width=\textwidth,height=6cm]{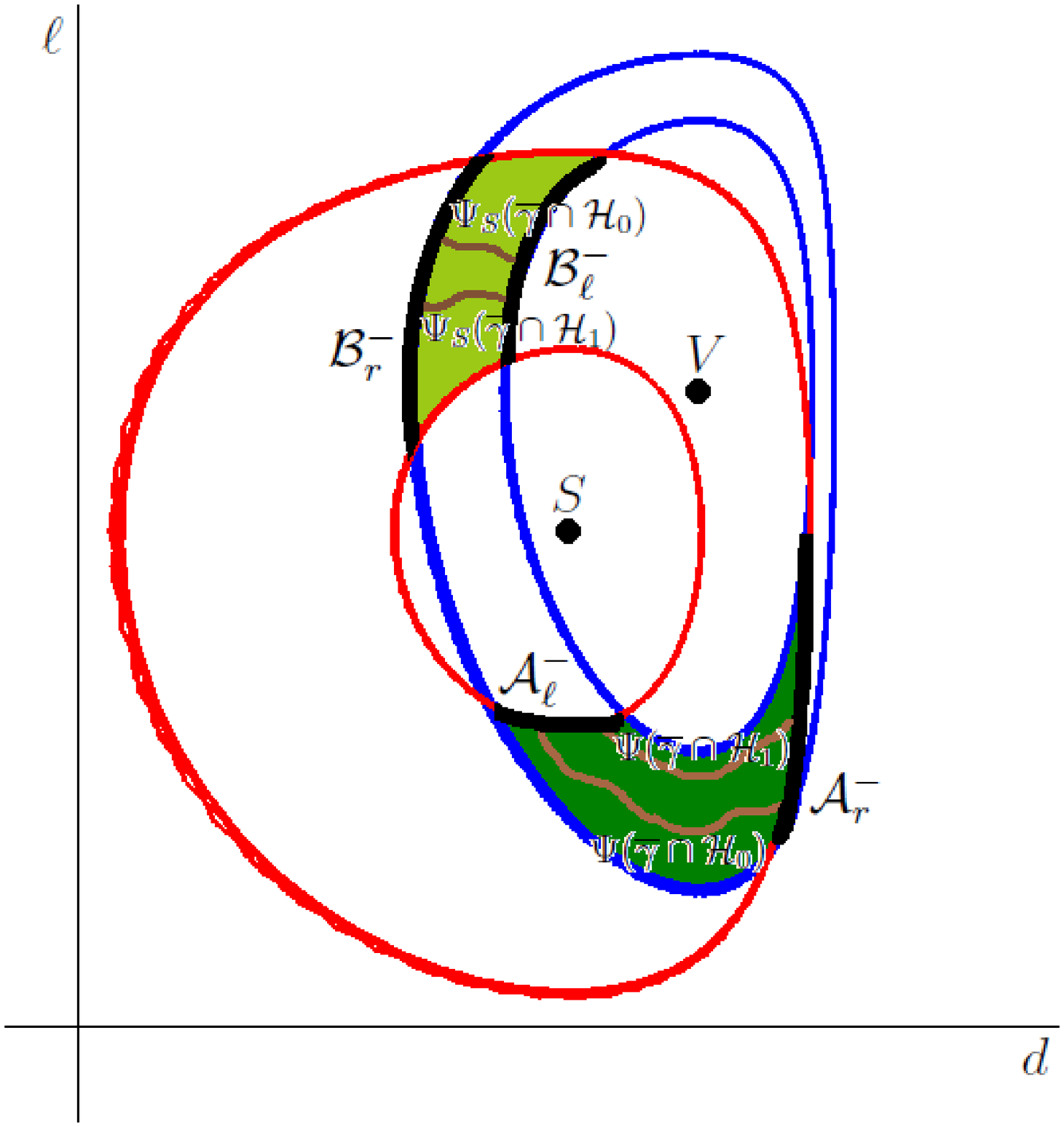}
\center{(B)}
\end{minipage}
\caption{In (A) we provide a qualitative representation of what happens when the stretching relation in \eqref{sr} is fulfilled. Namely, for any path $\gamma$ (in purple) joining $\mathcal A^{-}_{l}$ to $\mathcal A^{-}_{r}$ in $\mathcal A,$ there exist two compact subsets $\mathcal H_0$ and $\mathcal H_1$ (in lilac) of $\mathcal A$ such that the restriction (represented in brown) of $\gamma$ to each of them is transformed by $\Psi_S$ into a path (represented in brown) joining $\mathcal B^{-}_{l}$ to $\mathcal B^{-}_{r}$ in $\mathcal B.$ In (B)
we provide a qualitative representation of the effect produced on each of those two paths by the 
stretching relation $\displaystyle{\Psi_V: {\widetilde{\mathcal B}}\stretchx\, {\widetilde{\mathcal A}}}.$ 
In this case the Poincar\'e map $\Psi_V$ transforms those paths into paths (hazel colored) joining $\mathcal A^{-}_{l}$ to $\mathcal A^{-}_{r}$ in $\mathcal A.$}\label{sapf}
\end{figure}

As it is clear from the proof of Theorem \ref{app18}, the chaotic dynamics on two symbols are the result of the twist property
for the rotation number, generated by the different speeds with which the inner and the outer boundaries of two linked together annuli are run. Namely, for large enough time-intervals this suffices to make the image of paths through the Poincar\'e maps related to Systems 
\eqref{18b} and \eqref{18p} turn in a spiral-like fashion inside the annuli and cross a suitably high number of times the intersection sets between the linked annuli, where the invariant chaotic sets are located. See Figure \ref{sapf} for a qualitative graphical representation
of our framework, as well as Figures 4 and 5 in \cite{PaZa-08} for a similar scenario, in which the spiral effect is clearly illustrated. Indeed, if the switching times $T_S$ and $T_V$ were still larger, it could be proven that the Poincar\'e map $\Psi=\Psi_V\circ\Psi_S$ associated with System \eqref{18t} induces chaotic dynamics on $m$ symbols, where $m$ is any number greater or equal to $2,$ implying the existence of a semi-conjugacy between $\Psi$ and $\Sigma_m:=\{0,1,\dots,m-1\}^{\mathbb Z}$ and, consequently, implying the estimate $h_{\rm top}(\Psi)\ge\log(m)$ for the topological entropy. See Lemma 3.1 and Theorem 1.2 in \cite{PiZa-08} for further details.\\
We also stress that Theorem \ref{app18} is robust with respect to small perturbations in the coefficients of System \eqref{18t}, in $L^1$ norm. Namely, if $T_S$ and $T_V$ satisfy the conditions described in its statement, then, recalling the definition of System \eqref{18t} and of its coefficients in \eqref{xt}, there exists a positive constant $\varepsilon$ such that the same conclusions of Theorem \ref{app18} hold true for the system 
\begin{equation*}
\left\{
\begin{array}{ll}
d\,'=d(1-d)\left(\tilde\zeta(t)-\tilde\eta(t)\ell\right)\\
\vspace{-2mm}\\
\ell\,'=\ell(1-\ell)\left(-\tilde\vartheta(t)+\tilde\kappa(t) d\right)
\end{array}
\right. 
\end{equation*}
with $\tilde\zeta,\,\tilde\eta,\,\tilde\vartheta,\,\tilde\kappa: \mathbb R\to\mathbb R$ that are $T$-periodic functions,
as long as
$$
\int_0^T \bigl|\,\tilde\zeta(t)-\zeta\bigr|\,dt < \varepsilon,\quad
\int_0^T \bigl|\,\tilde\eta(t)-\eta(t)\bigr|\,dt < \varepsilon,
$$
$$
\int_0^T \bigl|\,\tilde\vartheta(t)-\vartheta(t)\bigr|\,dt < \varepsilon,\quad
\int_0^T \bigl|\,\tilde\kappa(t)-\kappa(t)\bigr|\,dt < \varepsilon.
$$
All the above remarks apply, with the suitable modifications, to the frameworks discussed in Subsection \ref{32} and in Section \ref{sec-4}, too.\\
We conclude the present investigation concerning the model proposed in \cite{Anea-18} by illustrating a
numerical example, in which Theorem \ref{app18} can be used to show the existence of chaotic dynamics.
\begin{example}\label{18ex}
Taking $p_D=0.9,\,p_{ND}=0.1,\,q_D=0.1,\,q_{ND}=0.2<\widehat{q}_{ND}=0.3,\,B^{PH}=6,\,E=140,\,C_L=90,$ we obtain 
System \eqref{18b} with $\zeta=B^{PH}=6,\,\eta=(p_Dq_D-p_{ND}q_{ND})E=9.8,\,\vartheta=q_{ND}(C_L-p_{ND}E)=15.2,\,\kappa=q_{ND}(C_L-p_{ND}E)+q_D(p_{D}E-C_L)=18.8,$ as well as System \eqref{18p} with $\zeta=B^{PH}=6,\,\widehat\eta=(p_Dq_D-p_{ND}\widehat{q}_{ND})E=8.4,\,\widehat\vartheta=\widehat{q}_{ND}(C_L-p_{ND}E)=22.8,\,\widehat\kappa=\widehat{q}_{ND}(C_L-p_{ND}E)+q_D(p_{D}E-C_L)=26.4.$
Hence, the former system has a center in $S=\left(\frac{\vartheta}{\kappa},\frac{\zeta}{\eta}\right)=(0.809,0.612),$ while the latter system has a center in $V=\left(\frac{\widehat\vartheta}{\widehat\kappa},\frac{\zeta}{\widehat\eta}\right)=(0.864,0.714).$\\
As shown in Figure \ref{nef}, two linked together annuli $\mathcal C_S(e_1,e_2)$ and $\mathcal C_V(h_1,h_2)$ can be obtained\,\footnote{We remark that, with respect to Figures \ref{la}--\ref{sapf}, Figure \ref{nef} is drawn for a different (more sensible, from an interpretative viewpoint) parameter configuration. However, in Figure \ref{nef} orbits are more squeezed, especially for System \eqref{18p}, so that the sets specified in the previous pictures would have been more difficult to highlight in the present context. Namely, Figures \ref{la}--\ref{sapf} just illustrate Definition \ref{li} and Theorem \ref{app18}, but we do not assign to them any interpretative value.} for $e_1=16.9,\,e_2=18.5,\,h_1=16.9,\,h_2=18.4,$ intersecting in the two disjoint generalized rectangles denoted by $\mathcal A$ and $\mathcal B.$ In this case, software-assisted computations show that 
$\tau_S(e_1)\approx 3.2 <\tau_S(e_2)\approx 3.8$ and $\tau_V(h_1)\approx 3<\tau_V(h_2)\approx 3.3.$ Hence, Theorem \ref{app18} guarantees the existence of chaotic dynamics for the Poincar\'e map $\Psi=\Psi_V\circ\Psi_S$ associated with System \eqref{18t} provided that 
$T_S>\frac{11\,\tau_S(e_1)\tau_S(e_2)}{2\,(\tau_S(e_2)-\tau_S(e_1))}\approx 111.467$ and 
$T_V>\frac{9\,\tau_V(h_1)\tau_V(h_2)}{2\,(\tau_V(h_2)-\tau_V(h_1))}\approx 148.500.$ Thus, considering $T=T_S+T_V=365,$ i.e., a period of one year, and assuming for instance that $T_S=T_V,$ i.e., that across the year the hot and the cold days are equally distributed, in which, according e.g. to \cite{Atea-08,Deea-12}, hypertension phenomena fall and raise, respectively, then it is possible to apply Theorem \ref{app18} to conclude that $\Psi$ induces chaotic dynamics on two symbols in $\mathcal A$ and $\mathcal B.$
\end{example}

\begin{figure}[ht]
\centering
\includegraphics[width=5cm,height=6cm]{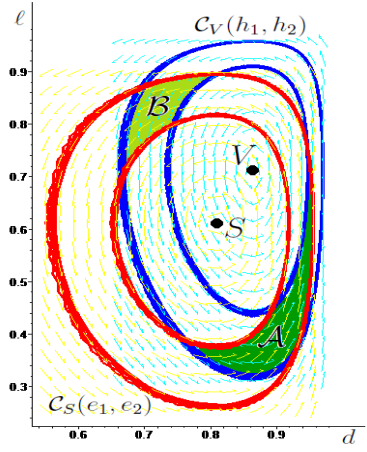}
\caption{The two linked together annuli $\mathcal C_S(e_1,e_2)$ (whose boundary is colored in red) and $\mathcal C_V(h_1,h_2)$ (whose boundary is colored in blue) considered in Example \ref{18ex}, together with the corresponding phase portrait.}\label{nef}
\end{figure}

\subsection{Analysis of the framework in Antoci et al. (2016)}\label{32}
Recalling that
$E_D=q_D R-(1-q_D)K$ and $E_{ND}=q_{ND}R-(1-q_{ND})K,$
when setting 
$\lambda=E_{ND}-E_D,\,\mu=C_D-C_{ND},\,\nu=E_{ND}-C_L,$
System \eqref{16} becomes
\begin{equation}\label{16b}
\left\{
\begin{array}{ll}
d\,'=d(1-d)\left(p\lambda\ell-\mu\right)\\
\vspace{-2mm}\\
\ell\,'=\ell(1-\ell)p\left(\nu-\lambda d\right)
\end{array}
\right. 
\end{equation}
with $d(t),\,\ell(t)\in[0,1].$ Focusing on the case $0<\mu<p\lambda$ and $0<\nu<\lambda,$ the center is given by 
$S=\left(\frac{\nu}{\lambda},\frac{\mu}{p\lambda}\right)$ and thus if $p$ increases, reaching an higher value $\widehat p,$ calling $V=\left(\frac{\nu}{\lambda},\frac{\mu}{\widehat p\lambda}\right)$ the new equilibrium, it is a global center of the perturbed system
\begin{equation}\label{16p}
\left\{
\begin{array}{ll}
d\,'=d(1-d)\left(\widehat p\lambda\ell-\mu\right)\\
\vspace{-2mm}\\
\ell\,'=\ell(1-\ell)\widehat p\left(\nu-\lambda d\right)
\end{array}
\right. 
\end{equation}
as long as $0<\mu<\widehat p\lambda$ and $0<\nu<\lambda.$
In particular, due to the raise in $p,$ the abscissa of $V$ coincides with that of $S,$ while the ordinate of $V$ is smaller than that of $S.$\\
In what follows, we describe just the main steps that allow to state and prove a result analogous to Theorem \ref{app18}, focusing on the differences with Subsection \ref{31}. In order to simplify the comparison between Subsections \ref{31} and \ref{32}, we will maintain the notation introduced therein to denote similar objects.\\
The orbits for System \eqref{16b} surrounding $S$ have equation
\begin{equation}\label{hs-16}
H_S(d,\ell)=-p\nu\log(d)+p(\nu-\lambda)\log(1-d)-\mu\log(\ell)+(\mu-p\lambda)\log(1-\ell)=e,
\end{equation} 
for some $e\ge e_0,$ where $e_0:=-p\nu\log(\frac{\nu}{\lambda})+p(\nu-\lambda)\log(1-\frac{\nu}{\lambda})-\mu\log(\frac{\mu}{p\lambda})+(\mu-p\lambda)\log(1-\frac{\mu}{p\lambda})$ is the minimum energy level attained by $H_S(d,\ell)$ on the unit square.
The orbits for System \eqref{16p} surrounding $V$ have equation
\begin{equation}\label{hv-16}
H_V(d,\ell)=-\widehat p\nu\log(d)+\widehat p(\nu-\lambda)\log(1-d)-\mu\log(\ell)+(\mu-\widehat p\lambda)\log(1-\ell)=h,
\end{equation}
for some $h\ge h_0,$ where $h_0:=H_V(V)$ is the minimum energy level attained by $H_V(d,\ell)$
on the unit square.\\
In the present framework, the straight line $r$ joining $S$ and $V$ is vertical, its equation being $d=\frac{\nu}{\lambda}.$ In order for Definition \ref{li} of linked together annuli to be still valid without the need of any change when considering annuli composed of level lines of $H_S$ in \eqref{hs-16} and of $H_V$ in \eqref{hv-16}, we define the ordering ``$\vartriangleleft$'' as a reversed comparison between the ordinate coordinates, 
 so that given $P=(d_P,\ell_P)$ and $Q=(d_Q,\ell_Q)$ belonging to $r,$ we have that $d_P=d_Q$ and it holds that $P\vartriangleleft\, Q$ (resp. $P\trianglelefteq\, Q$) if and only if $\ell_P>\ell_Q$ (resp. $\ell_P\ge\ell_Q$). See Figure \ref{16-la} for a graphical illustration of the new geometrical configuration.

\begin{figure}[ht]
\centering
\begin{minipage}[t]{0.48\textwidth}
\includegraphics[width=\textwidth,height=6cm]{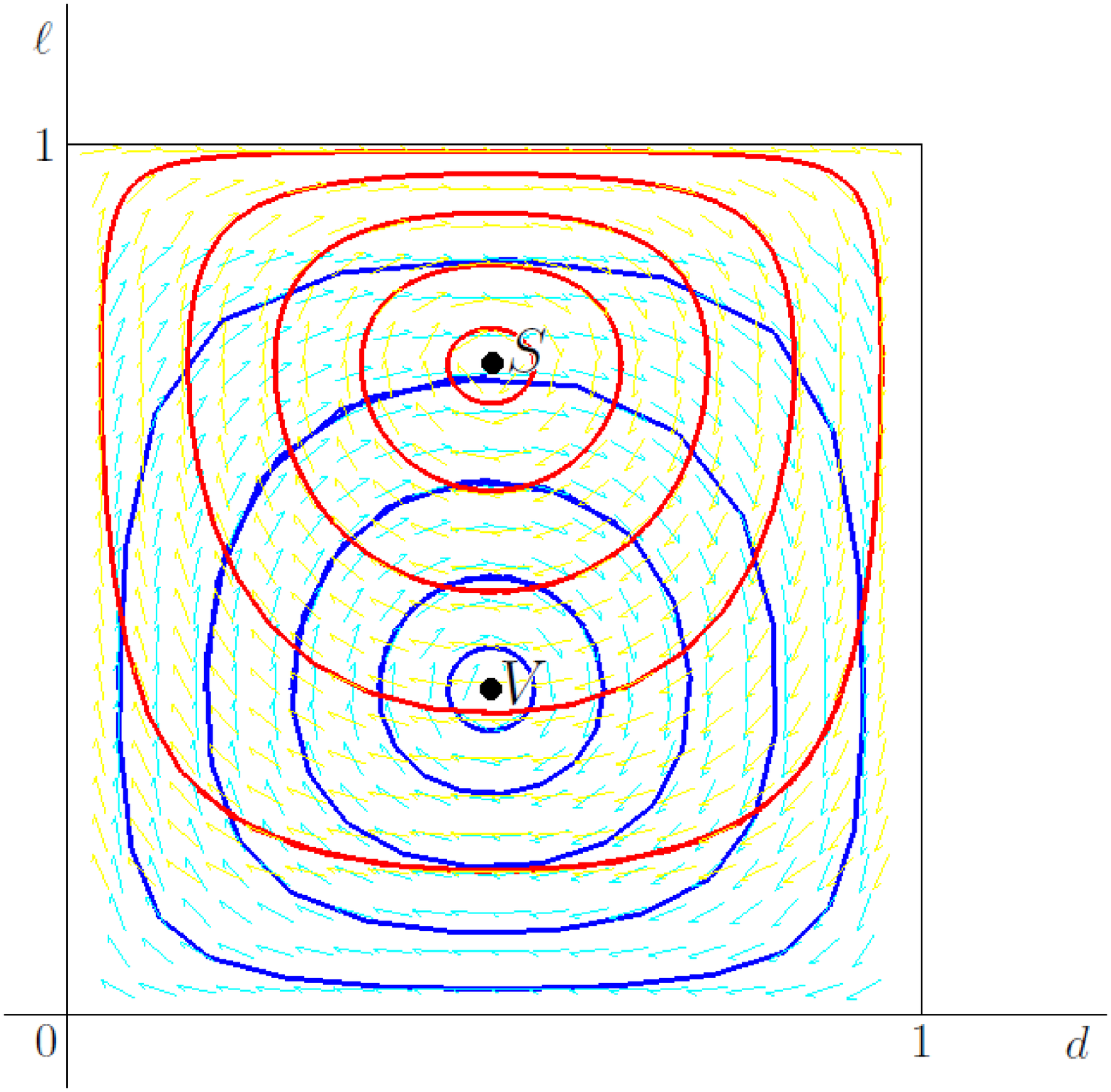}
\center{(A)}
\end{minipage}
\begin{minipage}[t]{0.48\textwidth}
\includegraphics[width=\textwidth,height=6cm]{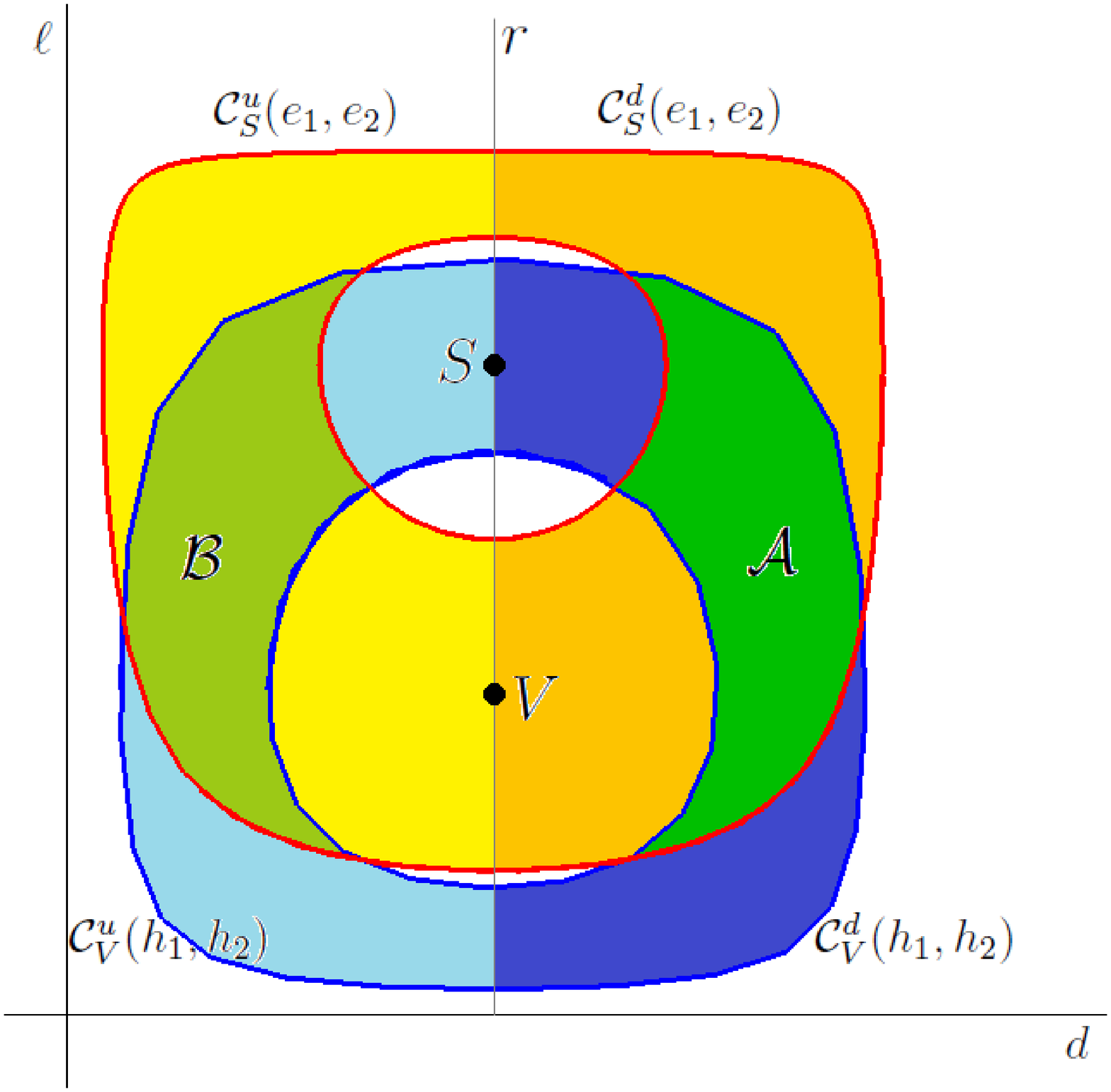}
\center{(B)}
\end{minipage}
\caption{In (A) we represent some energy level lines associated with System \eqref{16b}, in red, surrounding $S,$ as well as some energy level lines associated with System \eqref{16p}, in blue, surrounding $V.$ All of them are run clockwise. In (B), suitably choosing two level lines for each system, 
we obtain two linked together annuli, according to Definition \ref{li}. Calling them $\mathcal C_S(e_1,e_2)$ and $\mathcal C_V(h_1,h_2),$ we represent in yellow the set $\mathcal C_S^u(e_1,e_2)$ and in light blue the set $\mathcal C_V^u(h_1,h_2),$ both lying on the left of the vertical line $r,$ joining the centers $S$ and $V,$ while we represent in orange the set $\mathcal C_S^d(e_1,e_2)$ and in dark blue the set $\mathcal C_V^d(h_1,h_2),$ both lying on the right of $r.$ Notice that $\mathcal C_S^d(e_1,e_2)$ and $\mathcal C_V^d(h_1,h_2)$ cross in $\mathcal A$ (colored in dark green), and that $\mathcal C_S^u(e_1,e_2)$ and $\mathcal C_V^u(h_1,h_2)$ cross in $\mathcal B$ (colored in light green).}\label{16-la}
\end{figure}

Similar to Subsection \ref{31}, due to the seasonal variation in hypertension and the connection of the latter with perioperative adverse events, we can assume that $p$ in \eqref{16} alternate in a periodic fashion between a low and a high value, so that the dynamics are governed 
 by \eqref{16b} for $t\in[0,T_S)$, and by \eqref{16p} for $t\in[T_S,T_S+T_V).$ We suppose that the same alternation between the two regimes occurs with $T$-periodicity, where $T=T_S+T_V,$ i.e., we can assume that we are dealing with a system with periodic coefficients of the form
\begin{equation}\label{16t}
\left\{
\begin{array}{ll}
d\,'=d(1-d)\left(p(t)\lambda(t)\ell-\mu(t)\right)\\
\vspace{-2mm}\\
\ell\,'=\ell(1-\ell)p(t)\left(\nu(t)-\lambda(t) d\right)
\end{array}
\right. 
\end{equation}
where $x(t)\equiv x,$ for $x\in\{\lambda,\mu,\nu\},$ and for $p(t)$ it holds that
\begin{equation}\label{pt}
p(t)=\left\{
\begin{array}{ll}
p \quad & \mbox{for } \, t\in[0,T_S) \\
\vspace{-2mm}\\
\widehat p \quad & \mbox {for } \, t\in[T_S,T)
\end{array}
\right.
\end{equation} 
with $0<\mu<p\lambda$ and $0<\nu<\lambda.$ The function $p(t)$ is piecewise constant and it is supposed to be extended to the whole real
line by $T$-periodicity.\\
In regard to the Poincar\'e map $\Psi$ of System \eqref{16t}, it may be decomposed as $\Psi=\Psi_V\,\circ\Psi_S,$ where $\Psi_S$ is the Poincar\'e map associated with System \eqref{16b} for $t\in[0,T_S]$ and $\Psi_V$ is the Poincar\'e map associated with System \eqref{16p} for $t\in[0,T_V].$ Like for the framework analyzed in Subsection \ref{31}, the increasing monotonicity of the period of the orbits surrounding $S$ and $V,$ denoted respectively by $\tau_S(e)$ and $\tau_V(h),$ with the energy levels $e>e_0$ and $h>h_0,$ has been proven in \cite{Sc-85,Sc-90}. On the other hand, the analysis of the phase portrait shows that orbits surrounding $S$ and $V$ are both run clockwise this time. Since such dissimilarity with respect to the model proposed in \cite{Anea-18} will urge us to modify, with respect to Theorem \ref{app18}, some details in the proof of our result on the existence of chaotic dynamics for System \eqref{16t} (compare in particular the definitions of rotation number in \eqref{rns} and \eqref{rnv} with those in \eqref{rnsb} and \eqref{rnvb}, as well the estimates following from them), we present its statement in a precise manner, together with the main steps in its proof:
\begin{theorem}\label{app16}
For any choice of the parameters $0<\mu<p\lambda$ and $0<\nu<\lambda,$ defined like in System \eqref{16b}, and for any increase in the parameter $p,$ given the annulus $\mathcal C_S(e_1,e_2)$ around $S,$ for some $e_0<e_1<e_2,$ and the annulus $\mathcal C_V(h_1,h_2)$ around $V,$ for some $h_0<h_1<h_2,$ assume that they are linked together, calling $\mathcal A$ and $\mathcal B$ the connected components of $\mathcal C_S(e_1,e_2)\cap\mathcal C_V(h_1,h_2).$
Then, if $T_S>\frac{11\,\tau_S(e_1)\tau_S(e_2)}{2\,(\tau_S(e_2)-\tau_S(e_1))}$ and 
$T_V>\frac{9\,\tau_V(h_1)\tau_V(h_2)}{2\,(\tau_V(h_2)-\tau_V(h_1))},$ the 
Poincar\'e map $\Psi=\Psi_V\circ\Psi_S$ of System \eqref{16t} induces chaotic dynamics on two symbols in $\mathcal A$ and $\mathcal B,$ and thus
all the properties listed in Theorem \ref{th} are fulfilled for $\Psi$.
\end{theorem}
\begin{proof}
Given the linked together annuli $\mathcal C_S(e_1,e_2)$ and $\mathcal C_V(h_1,h_2),$ let us set $\mathcal C_S(e_1,e_2)=\mathcal C_S^u(e_1,e_2)\cup\mathcal C_S^d(e_1,e_2)$ and $\mathcal C_V(h_1,h_2)=\mathcal C_V^u(h_1,h_2)\cup\mathcal C_V^d(h_1,h_2),$ where $\mathcal C_S^u(e_1,e_2)$ (resp. $\mathcal C_S^d(e_1,e_2)$) is the subset of $\mathcal C_S(e_1,e_2)$ which lies on the left (resp. on the right) of the vertical line $r,$ joining $S$ and $V,$ and, analogously, $\mathcal C_V^u(h_1,h_2)$ (resp. $\mathcal C_V^d(h_1,h_2)$) is the subset of $\mathcal C_V(h_1,h_2)$ which lies on the left (resp. on the right) of the vertical line $r.$ See Figure \ref{16-la} (B) for a graphical illustration.\footnote{We stress that in this case the specifications $u$ and $d,$ staying respectively for ``up'' and ``down'', are meant with respect to the horizontal axis, before performing the rototranslation of $\mathbb R^2$ that brings the origin to the point $S$ and makes the horizontal axis coincide with the vertical line $r.$ As we shall see below, we need that rototranslation in order to introduce a system of generalized polar coordinates centered at $S,$ in view of defining the rotation number in \eqref{rnsb}. We will then perform the same operation with respect to $V,$ as well.}
Let us also set $\mathcal A:=\mathcal C_S^d(e_1,e_2)\cap \mathcal C_V^d(h_1,h_2)$ and $\mathcal B:=\mathcal C_S^u(e_1,e_2)\cap\mathcal C_V^u(h_1,h_2).$ We are going to show that, when orienting them by setting ${\mathcal A}^-=\mathcal A^{-}_{l}\cup\mathcal A^{-}_{r}$ and ${\mathcal B}^-=\mathcal B^{-}_{l}\cup\mathcal B^{-}_{r},$ with $\mathcal A^{-}_{l}:=\mathcal A\cap\Gamma_S(e_1),\,\mathcal A^{-}_{r}:=\mathcal A\cap\Gamma_S(e_2),\,\mathcal B^{-}_{l}:=\mathcal B\cap\Gamma_V(h_1),\,\mathcal B^{-}_{r}:=\mathcal B\cap\Gamma_V(h_2),$ then there exist two disjoint compact subsets ${\mathcal H_0},\,{\mathcal H_1}$ of ${\mathcal A}$
such that $\displaystyle{({\mathcal H}_i,\Psi_S): {\widetilde{\mathcal A}}\stretchx\, {\widetilde{\mathcal B}}},$ for $i=0,1,$ and that
$\displaystyle{\Psi_V: {\widetilde{\mathcal B}}\stretchx\, {\widetilde{\mathcal A}}}.$ 
If this is true, the Poincar\'e map $\Psi=\Psi_V\circ\Psi_S$ of System \eqref{16t} induces chaotic dynamics on two symbols in $\mathcal A$ 
and thus, since $\Psi$ is a homeomorphism on the open unit square, all the properties listed in Theorem \ref{th} are fulfilled for $\Psi$.\\
In order to check the former stretching relation, we introduce a system of generalized polar coordinates centered at $S,$ as explained in the proof of Theorem \ref{app18}, assuming to have performed the rototranslation of $\mathbb R^2$ that brings the origin to the point $S$ and makes the horizontal axis coincide with the vertical line $r.$ Since orbits for System \eqref{16b} are run clockwise, in order to count 
positive the turns around $S$ in the clockwise sense, we define the rotation number, describing the normalized angular displacement during the time interval $[0,t]\subseteq [0,T_S]$ of the solution 
$\varsigma_S(t\,,(d_0,\ell_0))$ to System \eqref{16b} with initial point in $(d_0,\ell_0)\in (0,1)^2,$ as 
\begin{equation}\label{rnsb}
\Rot_S(t,(d_0,\ell_0)):=\frac{\theta_S(0,(d_0,\ell_0))-\theta_S(t,(d_0,\ell_0))}{2\pi}\,.
\end{equation}
For $e>e_0,$ as a consequence of the star-shapedness with respect to $S$ of the lower contour sets $\{(d,\ell)\in (0,1)^2:H_S(d,\ell)\le e\},$
 with $H_S$ as in \eqref{hs-16}, conditions in \eqref{rot} are satisfied by $\Rot_S,$ too.
Let us now consider a generic path $\gamma:[0,1]\to{\mathcal A}$ with $\gamma(0)\in\mathcal A^{-}_{l},\,\gamma(1)\in\mathcal A^{-}_{r}$ and let us check that $T_S\ge\frac{11\,\tau_S(e_1)\tau_S(e_2)}{2\,(\tau_S(e_2)-\tau_S(e_1))}$ implies that $\theta_S(T_S,\gamma(1))-\theta_S(T_S,\gamma(0))>5\pi.$ Namely, like in \eqref{rots}, for every $t> 0$ it holds that 
$\Rot_S(t,\gamma(0)) - \Rot_S(t,\gamma(1))> t\, \frac{\tau_S(e_2) - \tau_S(e_1)}{\tau_S(e_1)\,\tau_S(e_2)}-2,$ and thus, 
for $T_S\ge\frac{11\,\tau_S(e_1)\tau_S(e_2)}{2\,(\tau_S(e_2)-\tau_S(e_1))},$ we have
$\Rot_S(T_S,\gamma(0)) - \Rot_S(T_S,\gamma(1))>3.$
Hence, $\theta_S(T_S,\gamma(1)) - \theta_S(T_S,\gamma(0))>6\pi+\theta_S(0,\gamma(1))-\theta_S(0,\gamma(0)).$ Since 
$\theta_S(0,\gamma(0)),\,\theta_S(0,\gamma(1))\in [-\pi,0]$ because $\overline\gamma\subset\mathcal A:=\mathcal C_S^d(e_1,e_2)\cap\mathcal C_V^d(h_1,h_2),$ it holds that $\theta_S(0,\gamma(1))-\theta_S(0,\gamma(0))>-\pi,$ from which it follows that
$\theta_S(T_S,\gamma(1)) - \theta_S(T_S,\gamma(0))>5\pi.$ 
As a consequence, there exists $n^{\ast}\in\mathbb N\setminus\{0\}$ such that 
$[-2(n^{\ast}+1)\pi,-2(n^{\ast}+1)\pi+\pi]$ and $[-2n^{\ast}\pi,-2n^{\ast}\pi+\pi]$ are contained in the interval 
$\{\theta_S(T_S,\gamma(\lambda)):\lambda\in[0,1]\}.$ Hence, by Bolzano theorem, there exist two disjoint maximal intervals 
$[\lambda_0',\lambda_0''],\,[\lambda_1',\lambda_1'']$ of $[0,1]$ such that for $i\in\{0,1\}$ it holds that
$\{\theta_S(T_S,\gamma(\lambda)):\lambda\in[\lambda_i',\lambda_i'']\}\subseteq[-2(n^{\ast}+i)\pi,-2(n^{\ast}+i)\pi+\pi],$
with $\theta_S(T_S,\gamma(\lambda_i'))=-2(n^{\ast}+i)\pi$ and $\theta_S(T_S,\gamma(\lambda_i''))=-2(n^{\ast}+i)\pi+\pi.$
It is easy to check that we can then set 
${\mathcal H}_i:=\{(d_0,\ell_0)\in\mathcal A:\theta_S(T_S,(d_0,\ell_0))\in[-2(n^{\ast}+i)\pi,-2(n^{\ast}+i)\pi+\pi]\}$ for $i\in\{0,1\}$
in order to have the stretching relation $\displaystyle{({\mathcal H}_i,\Psi_S): {\widetilde{\mathcal A}}\stretchx\, {\widetilde{\mathcal B}}}$ satisfied.\\
In order to prove that $\displaystyle{\Psi_V: {\widetilde{\mathcal B}}\stretchx\, {\widetilde{\mathcal A}}},$ we introduce a system of generalized polar coordinates, now centered at $V,$ defining in particular the rotation number as 
\begin{equation}\label{rnvb}
\Rot_V(t,(d_0,\ell_0)):=\frac{\theta_V(0,(d_0,\ell_0))-\theta_V(t,(d_0,\ell_0))}{2\pi}\,,
\end{equation}
since orbits for System \eqref{16p} are still run clockwise.
Let $\sigma:[0,1]\to{\mathcal B}$ be any path such that $\sigma(0)\in\mathcal B^{-}_{l}$ and
$\sigma(1)\in\mathcal B^{-}_{r}.$ Then, similar to what proven for $\Rot_S,$ we have
$\Rot_V(t,\sigma(0)) - \Rot_V(t,\sigma(1))>t\, \frac{\tau_V(h_2) - \tau_V(h_1)}{\tau_V(h_1)\,\tau_V(h_2)}\, - 2,$ and therefore if $T_V\ge\frac{9\,\tau_V(h_1)\tau_V(h_2)}{2\,(\tau_V(h_2)-\tau_V(h_1))}$ it follows that $\Rot_V(T_V,\sigma(0)) - \Rot_V(T_V,\sigma(1))>2,$
from which $\theta_V(T_V,\sigma(1))-\theta_V(T_V,\sigma(0))>3\pi,$ because $\overline\sigma\subset\mathcal B:=\mathcal C_S^u(e_1,e_2)\cap\mathcal C_V^u(h_1,h_2),$ and thus $\theta_V(0,\sigma(0)),\,\theta_V(0,\sigma(1))\in [0,\pi],$ implying that $\theta_V(0,\sigma(1))-\theta_V(0,\sigma(0))>-\pi.$\\
Hence, there exists $n^{\ast\ast}\in\mathbb N\setminus\{0\}$ such that 
$[-2n^{\ast\ast}\pi+\pi,-2n^{\ast\ast}\pi+2\pi]\subset\{\theta_V(T_V,\sigma(\lambda)):\lambda\in[0,1]\}$ and, consequently, there exists 
$[\lambda',\lambda'']\subseteq [0,1]$ such that 
$\{\theta_V(T_V,\sigma(\lambda)):\lambda\in[\lambda',\lambda'']\}\subseteq[-2n^{\ast\ast}\pi+\pi,-2n^{\ast\ast}\pi+2\pi],$
with $\theta_V(T_V,\sigma(\lambda'))=-2n^{\ast\ast}\pi+\pi$ and $\theta_V(T_V,\sigma(\lambda''))=-2n^{\ast\ast}\pi+2\pi.$
This easily allows to verify the stretching relation 
$\displaystyle{\Psi_V: {\widetilde{\mathcal B}}\stretchx\, {\widetilde{\mathcal A}}}.$\\
Having proved that conditions $(C_F)$ and $(C_G)$ in Theorem \ref{th} are fulfilled for the oriented rectangles ${\widetilde{\mathcal A}} := ({\mathcal A},{\mathcal A}^-)$ and
${\widetilde{\mathcal B}} := ({\mathcal B},{\mathcal B}^-)$ with $F=\Psi_S$ and $G=\Psi_V,$ it follows that $\Psi=\Psi_V\circ\Psi_S$
induces chaotic dynamics on two symbols in $\mathcal A.$ Since $\Psi$ is a homeomorphism on the open unit square, it is injective on the set ${\mathcal H}:=({\mathcal H}_0\cup{\mathcal H}_1)\cap{\Psi_S}^{-1}(\mathcal B),$ and thus
all the properties listed in Theorem \ref{th} are fulfilled for $\Psi$.\\
In order to check that $\Psi$ induces chaotic dynamics on two symbols in $\mathcal B,$ too, we orientate $\mathcal A$ by setting $\mathcal A^{-}_{l}:=\mathcal A\cap\Gamma_V(h_1),\,\mathcal A^{-}_{r}:=\mathcal A\cap\Gamma_V(h_2),\,\mathcal B^{-}_{l}:=\mathcal B\cap\Gamma_S(e_1),\,\mathcal B^{-}_{r}:=\mathcal B\cap\Gamma_S(e_2),$ and we should
show that $\displaystyle{({\mathcal K}_i,\Psi_S): {\widetilde{\mathcal B}}\stretchx\, {\widetilde{\mathcal A}}}$ for suitable compact disjoint subsets ${\mathcal K}_i$ of $\mathcal B,$ for $i\in\{0,1\},$ as well as that $\displaystyle{\Psi_V: {\widetilde{\mathcal A}}\stretchx\, {\widetilde{\mathcal B}}}.$ The verification of the details is omitted.
Then, by Theorem \ref{th} it holds that $\Psi$ induces chaotic dynamics on two symbols in $\mathcal B,$ as desired, and the proof is complete. 
\end{proof}$\hfill\square$ 

We now provide a numerical example about the model proposed in \cite{Anea-16}, in which the existence of chaotic dynamics
can be shown by using Theorem \ref{app16}.

\begin{example}\label{16ex}
Taking $p=0.1<\widehat p=0.2,\,q_D=0.4,\,q_{ND}=0.6,\,R=130,\,K=70,\,C_D=18,\,C_{ND}=15,\,C_L=30,\,$ 
we obtain $E_D=q_D R-(1-q_D)K=10,\,E_{ND}=q_{ND} R-(1-q_{ND})K=50,$ and consequently
Systems \eqref{16b}, \eqref{16p} with $\lambda=E_{ND}-E_D=40,\,\mu=C_D-C_{ND}=3,\,\nu=E_{ND}-C_L=20.$
Since $0<\mu<p\lambda$ and $0<\nu<\lambda,$ System \eqref{16b} has a center in $S=\left(\frac{\nu}{\lambda},\frac{\mu}{p\lambda}\right)=(0.5,0.75),$ while System \eqref{16p} has a center in $V=\left(\frac{\nu}{\lambda},\frac{\mu}{\widehat p\lambda}\right)=(0.5,0.375).$ This is the framework represented in Figure \ref{16-la}.\\
In particular, as shown in Figure \ref{16-la} (B), two linked together annuli $\mathcal C_S(e_1,e_2)$ and $\mathcal C_V(h_1,h_2)$ can be obtained for $e_1=5.4,\,e_2=8.3,\,h_1=12.1,\,h_2=16.2,$ intersecting in the two disjoint generalized rectangles denoted by $\mathcal A$ and $\mathcal B.$ Software-assisted computations show that 
$\tau_S(e_1)\approx 7.8 <\tau_S(e_2)\approx 11.3$ and $\tau_V(h_1)\approx 3.6<\tau_V(h_2)\approx 4.5.$ Hence, Theorem \ref{app16} guarantees the existence of chaotic dynamics for the Poincar\'e map $\Psi=\Psi_V\circ\Psi_S$ associated with System \eqref{16t} provided that 
$T_S>\frac{11\,\tau_S(e_1)\tau_S(e_2)}{2\,(\tau_S(e_2)-\tau_S(e_1))}\approx 138.5$ and 
$T_V>\frac{9\,\tau_V(h_1)\tau_V(h_2)}{2\,(\tau_V(h_2)-\tau_V(h_1))}\approx 81.$ Thus, considering a period of one year, and assuming that the hot and the cold days are equally distributed across the year, together with their opposite effect on hypertension phenomena (see e.g. \cite{Atea-08,Deea-12}), so that $T_S=T_V=\frac{365}{2},$ it is then possible to apply Theorem \ref{app16} to conclude that $\Psi$ induces chaotic dynamics on two symbols in $\mathcal A$ and $\mathcal B.$ 
\end{example}

\section{A biological framework with intraspecific competition}\label{sec-4}

We now provide an ecological interpretation of the model in \eqref{zan}, connected with intraspecific competition and environmental carrying capacity in predator-prey models. Its precise formulation is given by 
\begin{equation}\label{bio}
\left\{
\begin{array}{ll}
x\,'=r_x\, x\,\left(1-\frac{x}{K_x}\right)(\alpha-\beta y)\\
\vspace{-2mm}\\
y\,'=r_y\, y\,\left(1-\frac{y}{K_y}\right)(-\gamma+\delta x)
\end{array}
\right. 
\end{equation}
with $x(t),\,y(t)>0$ representing the size of the prey and of the predator populations,
respectively, and $\alpha,\,\beta,\,\gamma,\,\delta,\,r_x,\,r_y,\,K_x,\,K_y$ positive constants.
In particular, if $0<\alpha<\beta$ and $0<\gamma<\delta$ all orbits are closed and periodic, surrounding the unique internal equilibrium $S=\left(\frac{\gamma}{\delta},\frac{\alpha}{\beta}\right),$ which is a center. With respect to the classical Lotka-Volterra model (cf. for instance \cite{PiZa-08}), the model in \eqref{bio} encompasses the logistic terms $r_x\, \left(1-\frac{x}{K_x}\right)$ and $r_y\, \left(1-\frac{y}{K_y}\right),$ which take into account intra-species interaction and the role of environmental resources. More precisely, the parameters $r_x$ and $r_y$ are called the intrinsic (or inherent) growth rate for the two populations, while the parameters $K_x$ and $K_y$ are called environmental carrying capacity for the two species, as they describe the maximum population size that the resources of the environment can carry over a period of time (see \cite{Beea-06} for further ecological details).
Such biological interpretation of the model has been briefly suggested in \cite{Haea-07}, where however the focus is on the celebrated growth-cycle model by Goodwin \cite{Go-67,Go-72}, describing the dynamics of the wage share of output and the employment proportion. 
Actually, Harvie et al. in \cite{Haea-07} propose a system of differential equations in which each variable has both a positive and a
negative effect on its own growth rate, while in \eqref{bio} we focus just on the latter, in agreement with the models analyzed in Section \ref{sec-3}.\footnote{Still for the sake of conformity with the previous section, where the state variables $d(t)$ and $\ell(t)$ could vary just in $[0,1],$ being population shares, in Examples \ref{bioer} and \ref{bioek} we will confine ourselves to carrying capacities not exceeding unity, although $x(t)$ and $y(t)$ in \eqref{bio} can assume any positive value. In this manner those examples could be interpreted also in economic terms, concerning the wage share of output and the employment proportion, in line with \cite{Haea-07}. Of course, it would be easy to modify Examples \ref{bioer} and \ref{bioek} in order to allow $x(t)$ and $y(t)$ to exceed unity.}\\
In the existing literature, in view of taking into account the changes over time of the habitat conditions and of the available resources (see e.g. \cite{Ayea-12} for an empirical study concerning salmonids), some authors have introduced periodic variations in carrying capacities and intrinsic growth rates in related settings. Namely, in \cite{Le-16} the focus in on the effect of a periodic variation in the value of the carrying capacity for a species described by the logistic equation, while in
\cite{NiGu-76} a numerical investigation of the effect of a periodic variation in the values of the intrinsic growth rate and of the carrying capacity for a species described by the logistic equation with a time delay is performed.
Moreover, a periodic variation in the carrying capacity has been considered in \cite{BoWa-07,CuHe-02} for the Beverton-Holt equation, related to the logistic equation, as well as in \cite{BoSt-15}, where also the intrinsic growth rate is assumed to vary in a periodic fashion in a quantum calculus version of the classical Beverton-Holt equation. In the same perspective, we can assume that $r_x$ and $r_y,$ as well as $K_x$ and $K_y,$ periodically alternate between two different values, e.g. due to a seasonal effect. This allows us to enter the setting of Linked Twist Maps and to apply Theorem \ref{th} to prove the existence of chaotic dynamics induced by the related Poincar\'e map, although neither the intrinsic growth rates nor the carrying capacities of the two populations affect the center position. Namely, as shown in different contexts e.g. in \cite{BuZa-09,PaZa-13}, a geometrical configuration connected with LTMs may be obtained even when the center position is not affected by a variation of a certain parameter, as long as the latter suitably influences the shape of the orbits. We show two such frameworks in Figures \ref{biofr} and 
\ref{biofk}. The former is related to Example \ref{bioer}, where we focus on a variation of the intrinsic growth rates, while the latter illustrates Example \ref{bioek}, where we will deal with periodic oscillations in the two populations carrying capacities.
For clarity's sake, differently from \cite{BoSt-15,NiGu-76}, we will consider the variation in the intrinsic growth rates separately from the variation in the carrying capacities, in order to better understand the effect produced by each parameter on the orbits shape. Comparing Example \ref{bioer} with Example \ref{bioek}, we notice that raising the value of the intrinsic growth rate or of the carrying capacity for population $i\in\{x,y\}$ stretches orbits along the direction of the coordinate $i$-axis.
On the other hand, increasing carrying capacities enlarges the region in which orbits may lie, while increasing intrinsic growth rates does not. This difference is not very apparent in the framework considered in Example \ref{bioek}, where we will confine ourselves to carrying capacities close to unity, that is the value considered in Example \ref{bioer}, as explained in Footnote 12.\\
Turning to more technical aspects, since orbits of System \eqref{bio} are run counterclockwise, the proof of Theorem \ref{appbio}, concerning the framework with a periodic variation in intrinsic growth rates, will look similar to that of Theorem \ref{app18}, and thus it will be omitted for brevity's sake. For the framework with a periodic variation in carrying capacities, we will not even report the statement of our result about the existence of chaotic dynamics, limiting ourselves to some comments concerning it. On the other hand, since the center position is not influenced by a change in intrinsic growth rates or in carrying capacities, but the shape of the orbits is,
notation and, most importantly, the definition of linked together annuli need to be modified accordingly and also the statement of Theorem \ref{appbio} will be affected by those changes, together with the numerical Examples \ref{bioer} and \ref{bioek}. We refer the reader to Subsection \ref{31} for the remaining unchanged points.\\
Focusing at first on a variation in intrinsic growth rates in \eqref{bio}, let us assume that they alternate between certain values for the two populations, denoted by $r_x^{(1)}$ and $r_y^{(1)},$ for $t\in[0,T^{(1)})$, and different values, denoted by $r_x^{(2)}$ and $r_y^{(2)},$ for $t\in[T^{(1)},T^{(1)}+T^{(2)}).$\footnote{We stress that, in such manner, we are supposing that the switches between the different values of the intrinsic growth rates occur at the same time for both populations $x$ and $y.$ This is also the case that we will
consider in Example \ref{bioer}, in which we assume that periodic simultaneous exchanges between the intrinsic growth rate values of the two populations occur. However, more general frameworks in which the switches between the different values of the intrinsic growth rates are not simultaneous for the two populations could be easily analyzed as well, considering four dynamical regimes, rather than two. The same remark applies to the framework describing a periodic variation in the value of carrying capacities and to Example \ref{bioek}.}
Supposing that the same alternation between the two regimes occurs with $T$-periodicity, where $T=T^{(1)}+T^{(2)},$ we can assume that we are dealing with a system with periodic coefficients of the form
\begin{equation}\label{biot}
\left\{
\begin{array}{ll}
x\,'=r_x(t)\, x\,\left(1-\frac{x}{K_x(t)}\right)(\alpha(t)-\beta(t) y)\\
\vspace{-2mm}\\
y\,'=r_y(t)\, y\,\left(1-\frac{y}{K_y(t)}\right)(-\gamma(t)+\delta(t) x)
\end{array}
\right. 
\end{equation}
where $c(t)\equiv c,$ for $c\in\{K_x,K_y,\alpha,\beta,\gamma,\delta\},$ and for $r_i(t),$ with $i\in\{x,y\},$ it holds that
\begin{equation}\label{rt}
r_i(t)=\left\{
\begin{array}{ll}
r_i^{(1)} \quad & \mbox{for } \, t\in[0,T^{(1)}) \\
\vspace{-2mm}\\
r_i^{(2)} \quad & \mbox {for } \, t\in[T^{(1)},T)
\end{array}
\right.
\end{equation} 
with $0<\alpha<\beta$ and $0<\gamma<\delta.$ The functions $r_i(t),$ with $i\in\{x,y\},$ are piecewise constant and they are supposed to be extended to the whole real line by $T$-periodicity.\\
As observed above, the center coincides with $S=\left(\frac{\gamma}{\delta},\frac{\alpha}{\beta}\right)$ both when the intrinsic growth rates assume values $r_x^{(1)}$ and $r_y^{(1)},$ and thus we are in the regime whose dynamics are governed by the system that we will call $(1),$ as well as when they assume values $r_x^{(2)}$ and $r_y^{(2)},$ and thus we are in the regime whose dynamics are governed by the system that we will call $(2).$ In the former case, the orbits have equation
$$
\begin{array}{lll}
H^{(1)}(x,y)&\!\!\!\! = &\!\!\!\!\frac{1}{r_y^{(1)}}\left(-\alpha\log(y)+(\alpha-\beta K_y)\log(K_y-y)\right)+\\
&&\!\!\!\!\frac{1}{r_x^{(1)}}\left(-\gamma\log(x)+(\gamma-\delta K_x)\log(K_x-x)\right)=e,
\end{array}
$$
for some $e\ge e_0^{(1)},$ while, in the latter case, the orbits have equation
$$
\begin{array}{lll}
H^{(2)}(x,y)&\!\!\!\! = &\!\!\!\!\frac{1}{r_y^{(2)}}\left(-\alpha\log(y)+(\alpha-\beta K_y)\log(K_y-y)\right)+\\
&&\!\!\!\!\frac{1}{r_x^{(2)}}\left(-\gamma\log(x)+(\gamma-\delta K_x)\log(K_x-x)\right)=h,
\end{array}
$$
for some $h\ge h_0^{(2)},$ where $e_0^{(1)}$ and $h_0^{(2)}$ are the minimum energy levels attained by $H^{(1)}(x,y)$ and $H^{(2)}(x,y)$ on the the open rectangle $(0,K_x)\times(0,K_y),$ respectively, i.e., $e_0^{(1)}=H^{(1)}(S)$ and $h_0^{(2)}=H^{(2)}(S).$\\
The sets $\Gamma^{(1)}(e)=\{(x,y)\in (0,K_x)\times(0,K_y):H^{(1)}(x,y)=e\},$ for $e>e_0^{(1)},$ and 
$\Gamma^{(2)}(h)=\{(x,y)\in (0,K_x)\times(0,K_y):H^{(2)}(x,y)=h\},$ for $h>h_0^{(2)},$ are simple closed curves surrounding $S.$
We call {\it annulus around} $S$ {\it for System} (1) any set $\mathcal C^{(1)}(e_1,e_2)=\{(x,y)\in (0,K_x)\times(0,K_y):e_1\le H^{(1)}(x,y)\le e_2\}$ with $e_0^{(1)}<e_1<e_2,$ whose inner boundary coincides with $\Gamma^{(1)}(e_1)$ and whose outer boundary coincides with $\Gamma^{(1)}(e_2).$ Similarly, we call {\it annulus around} $S$ {\it for System} (2) any set $\mathcal C^{(2)}(h_1,h_2)=\{(x,y)\in (0,K_x)\times(0,K_y):h_1\le H^{(2)}(x,y)\le h_2\}$ with 
$h_0^{(2)}<h_1<h_2,$ whose inner boundary coincides with $\Gamma^{(2)}(h_1)$ and whose outer boundary coincides with $\Gamma^{(2)}(h_2).$\\
In the present framework, when considering annuli composed of level lines of $H^{(1)}$ and $H^{(2)},$ the definition of linked together annuli can be easily given by looking at the intersection points between the annuli and the horizontal line $\ell$ passing through $S,$ having equation $y=\frac{\alpha}{\beta}.$ The ordering
``$\vartriangleleft$'' in this case simply consists in a comparison between the $x$-coordinates,\footnote{We stress that a comparison between the vertical coordinates, as described in Subsection \ref{32}, would work as well. However, the consideration of the horizontal line $\ell$ is useful also in view of introducing the system of generalized polar coordinates centered at $S,$ that we need in the proof of Theorem \ref{appbio}, here omitted due to its similarity with the proof of Theorem \ref{app18}.} so that, given $P=(x_P,y_P)$ and $Q=(x_Q,y_Q)$ belonging to $\ell,$ we have that $y_P=y_Q$ and it holds that $P\vartriangleleft\, Q$ (resp. $P\trianglelefteq\, Q$) if and only if $x_P<x_Q$ (resp. $x_P\le x_Q$). 
In particular, assume that, like in Example \ref{bioer}, it holds that $r_x^{(1)}>r_x^{(2)}$ and $r_y^{(1)}<r_y^{(2)},$ so that orbits are  compressed along the $x$-direction and stretched along the $y$-direction passing from System (1) to System (2).
Then, taking inspiration from Figure \ref{biofr}, we introduce the definition of linked together annuli for the considered setting as follows:\footnote{Definition \ref{lim} can be easily modified when $r_x^{(1)}<r_x^{(2)}$ and $r_y^{(1)}>r_y^{(2)},$ so that orbits are stretched along the $x$-direction and compressed along the $y$-direction passing from System (1) to System (2). In the other cases (i.e., $r_x^{(1)}<r_x^{(2)}$ and $r_y^{(1)}<r_y^{(2)},$ or $r_x^{(1)}>r_x^{(2)}$ and $r_y^{(1)}>r_y^{(2)}$) it is sometimes possible to find linked together annuli, according to the considered parameter configurations.
For brevity's sake, we will not deepen such point, focusing in what follows just on the case $r_x^{(1)}>r_x^{(2)}$ and $r_y^{(1)}<r_y^{(2)},$ also in view of Example \ref{bioer}.} 
\begin{definition}\label{lim}
Given the annuli $\mathcal C^{(1)}(e_1,e_2)$ and $\mathcal C^{(2)}(h_1,h_2)$ around $S,$ we say that they are linked together if
$$S^{\,(1)}_{2,-}\vartriangleleft\, S^{\,(1)}_{1,-}\trianglelefteq\, S^{\,(2)}_{2,-}\vartriangleleft\, S^{\,(2)}_{1,-}\triangleleft\, S^{\,(2)}_{1,+}\vartriangleleft\, S^{\,(2)}_{2,+}\trianglelefteq\, S^{\,(1)}_{1,+}\vartriangleleft\, S^{\,(1)}_{2,+}$$
where, for $i\in\{1,2\},$ $S^{\,(1)}_{i,-}$ and $S^{\,(1)}_{i,+}$ denote the intersection points\footnote{Notice that, for $e_i>e_0^{(1)}$ and $h_i>h_0^{(2)},\,i\in\{1,2\},$ the boundary sets $\Gamma^{(1)}(e_i)$ and $\Gamma^{(2)}(h_i)$ intersect the straight line $\ell$ in exactly two points because 
$\{(x,y)\in (0,K_x)\times(0,K_y):H^{(1)}(x,y)\le e\}$ and $\{(x,y)\in (0,K_x)\times(0,K_y):H^{(2)}(x,y)\le h\}$ are star-shaped for all $e>e_0^{(1)}$ and for every $h>h_0^{(2)},$ respectively.} 
between $\Gamma^{(1)}(e_i)$ and the straight line $\ell,$ with 
$S^{\,(1)}_{i,-}\vartriangleleft\, S\vartriangleleft\, S^{\,(1)}_{i,+},$ and, similarly, $S^{\,(2)}_{i,-}$ and $S^{\,(2)}_{i,+}$ denote the intersection points between $\Gamma^{(2)}(h_i)$ and $\ell,$ with $S^{\,(2)}_{i,-}\vartriangleleft\, S\vartriangleleft\, S^{\,(2)}_{i,+}.$
\end{definition}
\noindent
Under the maintained assumptions, looking also at Figure \ref{biofr}, we deduce that, when an annulus composed of level lines of $H^{(1)}$ is linked with an annulus composed of level lines of $H^{(2)},$ then those two annuli cross in four pairwise disjoint generalized rectangles, differently from the frameworks considered in Section \ref{sec-3}, where only two intersections occurred.
In particular, each generalized rectangle contains a chaotic set, when the switching times between Systems (1) and (2) are large enough. Namely, in such eventuality, the presence of complex dynamics may be proved by applying Theorem \ref{th} to any pair composed by two of those generalized rectangles, when dealing with the Poincar\'e maps associated with Systems (1) and (2). Indeed, the Poincar\'e map $\Psi$ of System \eqref{biot} may be decomposed as $\Psi=\Psi^{(2)}\,\circ\Psi^{(1)},$ where $\Psi^{(1)}$ is the Poincar\'e map associated with System (1) for $t\in[0,T^{(1)}]$ and $\Psi^{(2)}$ is the Poincar\'e map associated with System (2) for $t\in[0,T^{(2)}].$ Similar to the settings analyzed in Section \ref{sec-3}, the increasing monotonicity for both systems of the orbit period, denoted respectively by $\tau^{(1)}(e)$ and $\tau^{(2)}(h),$ with the energy levels $e>e_0^{(1)}$ and $h>h_0^{(2)},$ has been proven in \cite{Sc-85,Sc-90}.\\
We are now in position to state our result about System \eqref{biot} with $r_i(t)$ defined as in \eqref{rt} for $i\in\{x,y\}.$
\begin{theorem}\label{appbio}
For any choice of the positive parameters $\alpha,\,\beta,\,\gamma,\,\delta,\,K_x,\,K_y,$ with
$0<\alpha<\beta$ and $0<\gamma<\delta,$ let $\mathcal C^{(1)}(e_1,e_2)$ be an annulus around $S,$ for some $e_0^{(1)}<e_1<e_2,$ obtained when the intrinsic growth rates for the two populations assume values $r_x^{(1)}$ and $r_y^{(1)},$ respectively, and let $\mathcal C^{(2)}(h_1,h_2)$ be an annulus around $S,$ for some $h_0^{(2)}<h_1<h_2,$ obtained when the intrinsic growth rates assume values $r_x^{(2)}$ and $r_y^{(2)}.$
Assume that $r_x^{(1)}>r_x^{(2)},\,r_y^{(1)}<r_y^{(2)}$ and that $\mathcal C^{(1)}(e_1,e_2)$ and $\mathcal C^{(2)}(h_1,h_2)$ are linked together, calling $\mathcal R_j,\,j\in\{1,2,3,4\},$ the connected components of $\mathcal C^{(1)}(e_1,e_2)\cap\mathcal C^{(2)}(h_1,h_2).$
Then, if $T^{(1)}>\frac{9\,\tau^{(1)}(e_1)\tau^{(1)}(e_2)}{2\,(\tau^{(1)}(e_2)-\tau^{(1)}(e_1))}$ and 
$T^{(2)}>\frac{7\,\tau^{(2)}(h_1)\tau^{(2)}(h_2)}{2\,(\tau^{(2)}(h_2)-\tau^{(2)}(h_1))},$ the 
Poincar\'e map $\Psi=\Psi^{(2)}\circ\Psi^{(1)}$ of System \eqref{biot}, with $r_i(t)$ defined as in \eqref{rt} for $i\in\{x,y\},$ induces chaotic dynamics on two symbols in every $\mathcal R_j,\,j\in\{1,2,3,4\},$ and thus all the properties listed in Theorem \ref{th} are fulfilled for $\Psi$.
\end{theorem}
The proof is omitted because of its similarity with that of Theorem \ref{app18}. We just stress that the smaller constants in the lower bounds of 
$T^{(1)}$ and $T^{(2)}$ in the statement of Theorem \ref{appbio} with respect to those for $T_S$ and $T_V$ in Theorem \ref{app18} are due to the fact that it is here possible to be more precise with the estimates concerning the angular coordinate, since, after having introduced a system of generalized polar coordinates centered at $S,$ it holds that each generalized rectangle 
$\mathcal R_j,\,j\in\{1,2,3,4\},$ is contained in a single quadrant, rather than in a half-plane. For further details, see Section 4 in 
\cite{BuZa-09}, where the authors deal with a similar geometrical configuration. Moreover, as observed above, the number of chaotic sets increases with respect to the frameworks considered in Subsections \ref{31} and \ref{32}, due to the larger number of intersection sets between two linked together annuli. Let us illustrate the just described framework in the next example.

\begin{example}\label{bioer}
Taking $\alpha=16<\beta=32,\,\gamma=24<\delta=30,\,K_x=K_y=1,$ System \eqref{bio} has a center in $S=\left(\frac{\gamma}{\delta},\frac{\alpha}{\beta}\right)=(0.8,0.5),$ both with $r_x=r_x^{(1)}=2,\,r_y=r_y^{(1)}=0.5,$ and with $r_x=r_x^{(2)}=0.5,\,r_y=r_y^{(2)}=2.$
Since $r_x^{(1)}>r_x^{(2)}$ and $r_y^{(1)}<r_y^{(2)},$ the intrinsic growth rate of population $x$ is larger in the time interval $[0,T^{(1)}),$ while the intrinsic growth rate of population $y$ is larger for $t\in[T^{(1)},T).$ Moreover, in the considered case, a comparison between preys and predators shows that the intrinsic growth rate of population $x$ is larger than the one of population $y$ in the former time period, since $r_x^{(1)}>r_y^{(1)},$ and the situation is reversed in the latter time period, since $r_x^{(2)}<r_y^{(2)}.$
As shown in Figure \ref{biofr}, in the analyzed framework two linked together annuli $\mathcal C^{(1)}(e_1,e_2)$ and $\mathcal C^{(2)}(h_1,h_2)$ can be obtained for $e_1=53,\,e_2=56.9,\,h_1=43,\,h_2=44.4,$ intersecting in the four pairwise disjoint generalized rectangles denoted by $\mathcal R_j,\,j\in\{1,2,3,4\}.$ Software-assisted computations show that 
$\tau^{(1)}(e_1)\approx 1.055 <\tau^{(1)}(e_2)\approx 1.195$ and $\tau^{(2)}(h_1)\approx 1.06<\tau^{(2)}(h_2)\approx 1.095.$ Hence, Theorem \ref{appbio} guarantees the existence of chaotic dynamics for the Poincar\'e map $\Psi=\Psi^{(2)}\circ\Psi^{(1)}$ associated with System \eqref{biot} with $r_i(t)$ defined as in \eqref{rt} for $i\in\{x,y\}$ provided that 
$T^{(1)}>\frac{9\,\tau^{(1)}(e_1)\tau^{(1)}(e_2)}{2\,(\tau^{(1)}(e_2)-\tau^{(1)}(e_1))}\approx 40.52$ and 
$T^{(2)}>\frac{7\,\tau^{(2)}(h_1)\tau^{(2)}(h_2)}{2\,(\tau^{(2)}(h_2)-\tau^{(2)}(h_1))}\approx 116.07.$ Thus, considering a period of one year, and assuming, to a first approximation, that $T^{(1)}=T^{(2)}=\frac{365}{2},$ so that the length of the time interval in which preys have an higher intrinsic growth rate coincides with the length of the time interval in which predators have an higher intrinsic growth rate, it is possible to apply Theorem \ref{appbio} to conclude that $\Psi$ induces chaotic dynamics on two symbols in every $\mathcal R_j,\,j\in\{1,2,3,4\}.$
\end{example}

\begin{figure}[ht]
\centering
\includegraphics[width=5.5cm,height=6cm]{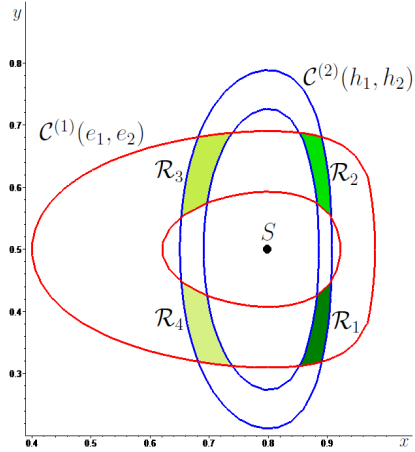}
\caption{The two linked together annuli $\mathcal C^{(1)}(e_1,e_2)$ (whose boundary is colored in red) and $\mathcal C^{(2)}(h_1,h_2)$ (whose boundary is colored in blue) considered in Example \ref{bioer}, together with the corresponding generalized rectangles $\mathcal R_j,\,j\in\{1,2,3,4\}.$}\label{biofr}
\end{figure}

Analogous considerations to those made above hold true when we deal with a variation in carrying capacities in \eqref{bio}, assuming that they alternate in a periodic fashion between certain values, denoted by $K_x^{(I)}$ and $K_y^{(I)},$ for $t\in[0,T^{(I)})$, and other values, denoted by $K_x^{(II)}$ and $K_y^{(II)},$ for $t\in[T^{(I)},T^{(I)}+T^{(II)}).$
Supposing that the same alternation between the two regimes occurs with $\mathcal T$-periodicity, where $\mathcal T=T^{(I)}+T^{(II)},$ we can assume that we are dealing with System \eqref{biot} with periodic coefficients,
where $c(t)\equiv c$ for $c\in\{r_x,r_y,\alpha,\beta,\gamma,\delta\}$ and for $K_i(t),$ with $i\in\{x,y\},$ it holds that
\begin{equation}\label{kt}
K_i(t)=\left\{
\begin{array}{ll}
K_i^{(I)} \quad & \mbox{for } \, t\in[0,T^{(I)}) \\
\vspace{-2mm}\\
K_i^{(II)} \quad & \mbox {for } \, t\in[T^{(I)},\mathcal T)
\end{array}
\right.
\end{equation} 
with $0<\alpha<\beta$ and $0<\gamma<\delta.$ The piecewise constant functions $K_i(t),$ with $i\in\{x,y\},$ are supposed to be extended to the whole real line by $\mathcal T$-periodicity. The center coincides with $S=\left(\frac{\gamma}{\delta},\frac{\alpha}{\beta}\right)$
 both when the carrying capacities assume values $K_x^{(I)}$ and $K_y^{(I)},$ and thus we are in the regime whose dynamics are governed by the system that we will call $(I),$ as well as when they assume values $K_x^{(II)}$ and $K_y^{(II)},$ and thus we are in the regime whose dynamics are governed by the system that we will call $(II).$ In the former case, the orbits have equation
$$
\begin{array}{lll}
H^{(I)}(x,y)&\!\!\!\! = &\!\!\!\!\frac{1}{r_y}\left(-\alpha\log(y)+(\alpha-\beta K_y^{(I)})\log(K_y^{(I)}-y)\right)+\\
&&\!\!\!\!\frac{1}{r_x}\left(-\gamma\log(x)+(\gamma-\delta K_x^{(I)})\log(K_x^{(I)}-x)\right)=e,
\end{array}
$$
for some $e\ge e_0^{(I)}=H^{(I)}(S),$ while, in the latter case, the orbits have equation
$$
\begin{array}{lll}
H^{(II)}(x,y)&\!\!\!\! = &\!\!\!\!\frac{1}{r_y}\left(-\alpha\log(y)+(\alpha-\beta K_y^{(II)})\log(K_y^{(II)}-y)\right)+\\
&&\!\!\!\!\frac{1}{r_x}\left(-\gamma\log(x)+(\gamma-\delta K_x^{(II)})\log(K_x^{(II)}-x)\right)=h,
\end{array}
$$
for some $h\ge h_0^{(II)}=H^{(II)}(S).$\\
Definition \ref{lim} of linked together annuli and the subsequent comments are valid for the present context, too, when adapting notation, by replacing (1) with $(I)$ and (2) with $(II).$ Also a result analogous to Theorem \ref{appbio} still hold true when considering variations in carrying capacities. We omit to report its precise statement, focusing instead on a numerical example in which it can be applied to show the existence of chaotic dynamics for the Poincar\'e map $\Psi=\Psi^{(II)}\circ\Psi^{(I)}$ associated with System \eqref{biot} with $K_i(t)$ defined as in \eqref{kt} for $i\in\{x,y\}.$

\begin{example}\label{bioek}
Taking $\alpha=64<\beta=128,\,\gamma=96<\delta=112,\,r_x=r_y=1,$ System \eqref{bio} has a center in $S=\left(\frac{\gamma}{\delta},\frac{\alpha}{\beta}\right)=(0.86,0.5),$ both with $K_x=K_x^{(I)}=0.99,\,K_y=K_y^{(I)}=0.9,$ and with $K_x=K_x^{(II)}=0.9,\,K_y=K_y^{(II)}=0.99.$
Similar to Example \ref{bioer}, since $K_x^{(I)}>K_x^{(II)}$ and $K_y^{(I)}<K_y^{(II)},$ the carrying capacity for population $x$ is larger in the time interval $[0,T^{(I)}),$ while the carrying capacity for population $y$ is larger for $t\in[T^{(I)},\mathcal T).$ Moreover, a comparison between preys and predators shows that the environment can carry a larger size of preys rather than of predators in the former time period, since $K_x^{(I)}>K_y^{(I)},$ and the situation is reversed in the latter time period, since $K_x^{(II)}<K_y^{(II)}.$
As shown in Figure \ref{biofk}, in the analyzed framework two linked together annuli $\mathcal C^{(I)}(e_1,e_2)$ and $\mathcal C^{(II)}(h_1,h_2)$ can be obtained for $e_1=145.3,\,e_2=146.7,\,h_1=129.4,\,h_2=131.1,$ intersecting in the four pairwise disjoint generalized rectangles denoted by $\mathcal R_j,\,j\in\{1,2,3,4\}.$ Software-assisted computations for the periods show that 
$\tau^{(I)}(e_1)\approx 0.355 <\tau^{(I)}(e_2)\approx 0.360$ and $\tau^{(II)}(h_1)\approx 0.635<\tau^{(II)}(h_2)\approx 0.650.$\footnote{Also in this setting, the increasing monotonicity of the period of the orbits follows from the results in \cite{Sc-85,Sc-90}.} Hence, a result analogous to Theorem \ref{appbio} guarantees the existence of chaotic dynamics for the Poincar\'e map $\Psi=\Psi^{(II)}\circ\Psi^{(I)}$ associated with System \eqref{biot} with $K_i(t)$ defined as in \eqref{kt} for $i\in\{x,y\}$ provided that 
$T^{(I)}>\frac{9\,\tau^{(I)}(e_1)\tau^{(I)}(e_2)}{2\,(\tau^{(I)}(e_2)-\tau^{(I)}(e_1))}\approx 115.02$ and 
$T^{(II)}>\frac{7\,\tau^{(II)}(h_1)\tau^{(II)}(h_2)}{2\,(\tau^{(II)}(h_2)-\tau^{(II)}(h_1))}\approx 96.31.$ Thus, considering a period of one year, and assuming e.g. that $T^{(I)}=T^{(II)}=\frac{365}{2},$ it is then possible to apply Theorem \ref{appbio} to conclude that $\Psi$ induces chaotic dynamics on two symbols in every $\mathcal R_j,\,j\in\{1,2,3,4\}.$
\end{example}

\begin{figure}[ht]
\centering
\includegraphics[width=5.5cm,height=6cm]{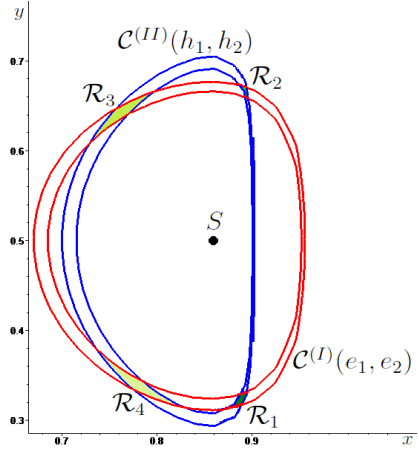}
\caption{The two linked together annuli $\mathcal C^{(I)}(e_1,e_2)$ (whose boundary is colored in red) and $\mathcal C^{(II)}(h_1,h_2)$ (whose boundary is colored in blue) considered in Example \ref{bioek}, together with the corresponding generalized rectangles $\mathcal R_j,\,j\in\{1,2,3,4\}.$}\label{biofk}
\end{figure}

\section{Conclusion}\label{sec-5}
In the present work, led by various possible interpretations of the same Hamiltonian system, we have shown how to apply the Linked Twist Maps (LTMs) method to prove the existence of chaotic dynamics for the associated Poincar\'e map, perturbing each time a different parameter in a periodic fashion and obtaining various geometrical configurations in the phase plane, when considering both the orbits of the original system and those of the perturbed one. 
Namely, we can try to use the LTMs technique whenever we have a conservative system with a nonisochronous center and it is sensible to assume a periodic variation for one or more model parameters\footnote{See e.g. \cite{Maea-10}, where the authors apply the LTMs method to a class of periodically perturbed
planar Hamiltonian systems, dealing also with perturbations by means of dissipative terms.}, both influencing the center or not, e.g. due to a seasonal effect. We just stress that, if the position of the center is not affected by the change in the considered parameter value, it is necessary to check that the shape of the orbits changes in a way that allows to detect at least two linked annuli.\\ 
We finally remark that in the present contribution we have proposed planar applications of the LTMs technique. Nonetheless, such method can be employed in (at least) three-dimensional frameworks, too, when dealing with linked together cylindrical sets. See \cite{RZZa-14} for a 3D non-Hamiltonian example concerning a predator-prey model in a periodically varying
environment. We will investigate a three-dimensional continuous-time model connected with some game theoretical context in a future work, taking inspiration e.g. from the 3D setting considered in \cite{Anea-19} (cf. Footnote 6).
In regard to economic settings, we recall that the SAP method, on which the LTMs technique is based, has been recently used to prove the existence of chaotic dynamics for some discrete-time triopoly game models in \cite{Pi-15,Pi-16}, while in \cite{Meea-09} the SAP technique had been applied to one- and two-dimensional discrete-time economic models, concerning overlapping generations and duopoly frameworks.

%
%

\section*{Acknowledgments}

\noindent
Many thanks to Professor J. Hofbauer for useful discussions about the monotonicity of the period of the orbits.

\newpage


\end{document}